\documentclass[reqno]{amsart}

\usepackage{amsmath,amsfonts,amssymb,amscd,verbatim,delarray,fullpage}

\usepackage{tikz}
\usepackage{tikz-cd}
\usetikzlibrary{matrix,arrows}
\usepackage{stmaryrd}
\usepackage{amsthm}
\usepackage{url}
\usepackage[polutonikogreek,english]{babel}

\DeclareSymbolFont{bbold}{U}{bbold}{m}{n}
\DeclareSymbolFontAlphabet{\mathbbold}{bbold}

\DeclareMathOperator{\fin}{fin.}
\DeclareMathOperator{\pot}{pot.}
\DeclareMathOperator{\pr}{pr}
\DeclareMathOperator{\Inn}{Inn}
\DeclareMathOperator{\red}{red}
\DeclareMathOperator{\Div}{div}
\DeclareMathOperator{\ab}{ab}
\DeclareMathOperator{\can}{can}

\DeclareMathOperator{\dic}{Dic}
\DeclareMathOperator{\nonab}{non-ab}
\DeclareMathOperator{\maxm}{max}
\DeclareMathOperator{\genus}{genus}

\DeclareMathOperator{\ns}{ns}
\DeclareMathOperator{\s}{s}
\DeclareMathOperator{\aut}{Aut}

\DeclareMathOperator{\tors}{tors}
\DeclareMathOperator{\GL}{GL}
\DeclareMathOperator{\SL}{SL}
\DeclareMathOperator{\PSL}{PSL}
\DeclareMathOperator{\PGL}{PGL}

\DeclareMathOperator{\borel}{B}

\chardef\bslash=`\\ 

\begin{document}


\newtheorem{Theorem}{Theorem}[section]

\newtheorem{example}[Theorem]{Example}
\newtheorem{cor}[Theorem]{Corollary}
\newtheorem{goal}[Theorem]{Goal}

\newtheorem{Conjecture}[Theorem]{Conjecture}
\newtheorem{guess}[Theorem]{Guess}

\newtheorem{exercise}[Theorem]{Exercise}
\newtheorem{Question}[Theorem]{Question}
\newtheorem{lemma}[Theorem]{Lemma}
\newtheorem{property}[Theorem]{Property}
\newtheorem{proposition}[Theorem]{Proposition}
\newtheorem{ax}[Theorem]{Axiom}
\newtheorem{claim}[Theorem]{Claim}

\newtheorem{nTheorem}{Surjectivity Theorem}

\theoremstyle{definition}
\newtheorem{Definition}[Theorem]{Definition}
\newtheorem{problem}[Theorem]{Problem}
\newtheorem{question}[Theorem]{Question}
\newtheorem{Example}[Theorem]{Example}

\newtheorem{remark}[Theorem]{Remark}
\newtheorem{diagram}{Diagram}
\newtheorem{Remark}[Theorem]{Remark}
\newcommand{\diagref}[1]{diagram~\ref{#1}}
\newcommand{\thmref}[1]{Theorem~\ref{#1}}
\newcommand{\secref}[1]{Section~\ref{#1}}
\newcommand{\subsecref}[1]{Subsection~\ref{#1}}
\newcommand{\lemref}[1]{Lemma~\ref{#1}}
\newcommand{\corref}[1]{Corollary~\ref{#1}}
\newcommand{\exampref}[1]{Example~\ref{#1}}
\newcommand{\remarkref}[1]{Remark~\ref{#1}}
\newcommand{\corlref}[1]{Corollary~\ref{#1}}
\newcommand{\claimref}[1]{Claim~\ref{#1}}
\newcommand{\defnref}[1]{Definition~\ref{#1}}
\newcommand{\propref}[1]{Proposition~\ref{#1}}
\newcommand{\prref}[1]{Property~\ref{#1}}
\newcommand{\itemref}[1]{(\ref{#1})}
\newcommand{\ul}[1]{\underline{#1}}


\newcommand{\CE}{\mathcal{E}}
\newcommand{\CG}{\mathcal{G}}\newcommand{\CV}{\mathcal{V}}
\newcommand{\CL}{\mathcal{L}}
\newcommand{\CM}{\mathcal{M}}
\newcommand{\A}{\mathcal{A}}
\newcommand{\CO}{\mathcal{O}}
\newcommand{\B}{\mathcal{B}}
\newcommand{\CS}{\mathcal{S}}
\newcommand{\CX}{\mathcal{X}}
\newcommand{\CY}{\mathcal{Y}}
\newcommand{\CT}{\mathcal{T}}
\newcommand{\CW}{\mathcal{W}}
\newcommand{\CJ}{\mathcal{J}}

\newcommand\myeq{\mathrel{\stackrel{\makebox[0pt]{\mbox{\normalfont\tiny def}}}
{\Longleftrightarrow}}}
\newcommand{\st}{\sigma}
\renewcommand{\k}{\varkappa}
\newcommand{\Frac}{\mbox{Frac}}
\newcommand{\XC}{\mathcal{X}}
\newcommand{\wt}{\widetilde}
\newcommand{\wh}{\widehat}
\newcommand{\mk}{\medskip}
\renewcommand{\sectionmark}[1]{}
\renewcommand{\Im}{\operatorname{Im}}
\renewcommand{\Re}{\operatorname{Re}}
\newcommand{\la}{\langle}
\newcommand{\ra}{\rangle}
\newcommand{\LND}{\mbox{LND}}
\newcommand{\Pic}{\mbox{Pic}}
\newcommand{\lnd}{\mbox{lnd}}
\newcommand{\GLND}{\mbox{GLND}}\newcommand{\glnd}{\mbox{glnd}}
\newcommand{\Der}{\mbox{DER}}\newcommand{\DER}{\mbox{DER}}
\renewcommand{\th}{\theta}
\newcommand{\ve}{\varepsilon}
\newcommand{\1}{^{-1}}
\newcommand{\iy}{\infty}
\newcommand{\iintl}{\iint\limits}
\newcommand{\capl}{\operatornamewithlimits{\bigcap}\limits}
\newcommand{\cupl}{\operatornamewithlimits{\bigcup}\limits}
\newcommand{\suml}{\sum\limits}
\newcommand{\ord}{\operatorname{ord}}
\newcommand{\gal}{\operatorname{Gal}}
\newcommand{\bk}{\bigskip}
\newcommand{\fc}{\frac}
\newcommand{\g}{\gamma}
\newcommand{\be}{\beta}
\newcommand{\dl}{\delta}
\newcommand{\Dl}{\Delta}
\newcommand{\lm}{\lambda}
\newcommand{\Lm}{\Lambda}
\newcommand{\om}{\omega}
\newcommand{\ov}{\overline}
\newcommand{\vp}{\varphi}
\newcommand{\kap}{\varkappa}

\newcommand{\Vp}{\Phi}
\newcommand{\Varphi}{\Phi}
\newcommand{\BC}{\mathbb{C}}
\newcommand{\C}{\mathbb{C}}\newcommand{\BP}{\mathbb{P}}
\newcommand{\BQ}{\mathbb {Q}}
\newcommand{\BM}{\mathbb{M}}
\newcommand{\mbh}{\mathbb{H}}
\newcommand{\BR}{\mathbb{R}}\newcommand{\BN}{\mathbb{N}}
\newcommand{\BZ}{\mathbb{Z}}\newcommand{\BF}{\mathbb{F}}
\newcommand{\BA}{\mathbb {A}}
\renewcommand{\Im}{\operatorname{Im}}
\newcommand{\idd}{\operatorname{id}}
\newcommand{\ep}{\epsilon}
\newcommand{\tp}{\tilde\partial}
\newcommand{\doe}{\overset{\text{def}}{=}}
\newcommand{\supp} {\operatorname{supp}}
\newcommand{\loc} {\operatorname{loc}}
\newcommand{\de}{\partial}
\newcommand{\z}{\zeta}
\renewcommand{\a}{\alpha}
\newcommand{\G}{\Gamma}
\newcommand{\der}{\mbox{DER}}

\newcommand{\Spec}{\operatorname{Spec}}
\newcommand{\Sym}{\operatorname{Sym}}
\newcommand{\Aut}{\operatorname{Aut}}

\newcommand{\Idd}{\operatorname{Id}}

\newcommand{\tG}{\widetilde G}
\newcommand{\F}{\mathbb{F}}
\newcommand{\Q}{\mathbb{Q}}
\newcommand{\Z}{\mathbb{Z}}
\newcommand{\XG}{\mc{E}_{S_4}(5)}
\newcommand{\tB}{\text{B}}
\newcommand{\Gal}{\text{Gal}}
\newcommand{\cX}{\mathcal{X}}
\newcommand{\bP}{\mathbf{P}}
\newcommand{\FX}{\mathfrac {X}}
\newcommand{\FV}{\mathfrac {V}}
\newcommand{\SX}{\mathcal {X}}
\newcommand{\SV}{\mathcal {V}}
\newcommand{\SO}{\mathcal {O}}
\newcommand{\SD}{\mathcal {D}}
\newcommand{\Sr}{\rho}
\newcommand{\SR}{\mathcal {R}}
\newcommand{\cl}{\mathcal{C}}
\newcommand{\ok}{\mathcal{O}_K}

\setcounter{equation}{0} \setcounter{section}{0}

\newcommand{\ds}{\displaystyle}
\newcommand{\gl}{\lambda}
\newcommand{\gL}{\Lambda}
\newcommand{\gge}{\epsilon}
\newcommand{\gG}{\Gamma}
\newcommand{\ga}{\alpha}
\newcommand{\gb}{\beta}
\newcommand{\gd}{\delta}
\newcommand{\gD}{\Delta}
\newcommand{\gs}{\sigma}
\newcommand{\mbq}{\mathbb{Q}}
\newcommand{\mbr}{\mathbb{R}}
\newcommand{\mbz}{\mathbb{Z}}
\newcommand{\mbc}{\mathbb{C}}
\newcommand{\mbn}{\mathbb{N}}
\newcommand{\mbp}{\mathbb{P}}
\newcommand{\mbf}{\mathbb{F}}
\newcommand{\mbe}{\mathbb{E}}
\newcommand{\lcm}{\text{lcm}\,}
\newcommand{\mf}[1]{\mathfrak{#1}}
\newcommand{\ol}[1]{\overline{#1}}
\newcommand{\mc}[1]{\mathcal{#1}}
\newcommand{\nequiv}{\equiv\hspace{-.07in}/\;}
\newcommand{\bnequiv}{\equiv\hspace{-.13in}/\;}

\title{Elliptic curves with non-abelian entanglements}
\author{Nathan Jones and Ken McMurdy}


\date{}

\begin{abstract}
We consider the problem of classifying quadruples $(K,E,m_1,m_2)$ where $K$ is a number field, $E$ is an elliptic curve defined over $K$ and $(m_1,m_2)$ is a pair of relatively prime positive integers for which the intersection $K(E[m_1]) \cap K(E[m_2])$ is a non-abelian extension of $K$.  There is an infinite set $\mc{S}$ of modular curves whose $K$ rational points capture all elliptic curves over $K$ without complex multiplication that have this property.  Our main theorem explicitly describes the (finite) subset of $\mc{S}$ consisting of those modular curves having genus zero.  In the case $K = \mbq$, this has applications to the problem of determining when the Galois representation on the torsion of $E$ is as large as possible modulo a prescribed obstruction; we illustrate this application with a specific example.
\end{abstract}

\maketitle

\section{Introduction} \label{introduction}

Let $K$ be a field of characteristic zero and $E$ an elliptic curve over $K$.  For a positive integer $m$, let $E[m]$ denote the $m$-torsion of $E$ and 
\[
K(E[m]) := K \left( \left\{ x, y \in \ol{K} : (x,y) \in E[m] \right\} \right)
\]
the $m$-th division field of $E$ over $K$, obtained by adjoining to $K$ the $x$ and $y$ coordinates of the $m$-torsion of some (any) Weierstrass model of $E$.  The restriction of $\gal(K(E[m])/K)$ to $E[m]$ gives rise to an embedding
\[
\gal(K(E[m])/K) \hookrightarrow \aut(E[m]) \simeq \GL_2(\mbz/m\mbz),
\]
the latter isomorphism induced by the choice of a $\mbz/m\mbz$-basis for $E[m]$, which is a free $\mbz/m\mbz$-module of rank $2$.  It is of interest to understand the image of this embedding as $m$ varies over all positive integers, for $K$ and $E$ fixed.  In the present paper, we are focused on the case where $m$ has more than one distinct prime factor.  Writing $m = m_1m_2$ where $\gcd(m_1, m_2) = 1$ and each $m_i$ is greater than $1$, we have
\[
\gal(K(E[m])/K) \subseteq \gal(K(E[m_1])/K) \times \gal(K(E[m_2])/K) \subseteq \GL_2(\mbz/m_1\mbz) \times \GL_2(\mbz/m_2\mbz).
\]
By Galois theory, the first inclusion is proper if and only if $K(E[m_1]) \cap K(E[m_2]) \neq K$.  In particular, understanding $\gal(K(E[m])/K)$ amounts to understanding each of the groups $\gal(K(E[m_1])/K)$, $\gal(K(E[m_2])/K)$ as well as the \emph{entanglement} $K(E[m_1]) \cap K(E[m_2])$, and ``how it sits'' inside $K(E[m_1])$ and $K(E[m_2])$.  In recent years, there has been significant interest in both the nature of division fields $K(E[m])$ for composite level $m$ (see for instance \cite{rousezureickbrown}, \cite{sutherlandzywina} and \cite{morrow}) and the nature of entanglements (see \cite{danielsmorrow}, \cite{campagnapengo}, \cite{danielslozanorobledo} and \cite{cnlmpzb}).
In the present paper, we are concerned with the following problem.
\begin{Definition}
Let $E$ be an elliptic curve defined over a field $K$ and let $m \in \mbn$ be a positive integer that is divisible by at least two primes.  We call a factorization $m = m_1m_2$ of $m$ {\bf{permissible}} if the factors $m_1$ and $m_2$ are co-prime and each greater than one.  Given a permissible factorization $m = m_1m_2$, we call the field extension $K \subseteq K(E[m_1]) \cap K(E[m_2])$ the {\bf{entanglement associated to $E/K$ and $(m_1,m_2)$}}.  We say that  $E$ {\bf{has a non-abelian entanglement over $K$ at level $m$}} if, for some permissible factorization $m = m_1m_2$, the entanglement associated to $E/K$ and $(m_1,m_2)$ is a non-abelian extension of $K$.  Finally, we say that {\bf{$E$ has a non-abelian entanglement over $K$}} if it has a non-abelian entanglement over $K$ at some level $m$.
\end{Definition}
\begin{Remark}
In case the pair $(m_1,m_2)$ is uniquely determined by $m$ (i.e. in case $m$ has exactly 2 prime factors), we call the extension $K \subseteq K(E[m_1]) \cap K(E[m_2])$ simply the {\bf{entanglement at $m$ associated to $E/K$}}.
\end{Remark}
\begin{problem} \label{entanglementproblem}
For a given number field $K$, classify the elliptic curves $E$ over $K$ that have a non-abelian entanglement over $K$.  (This is a restatement of \cite[Question 1.1]{braujones}.)
\end{problem}
It is difficult to address this problem completely, since non-abelian entanglements can correspond to $K$-rational points on modular curves of genus greater than $1$, and could thus occur ``sporadically'' for large $m$, a la Faltings' Theorem.  
We therefore focus at present on classifying all ``one-parameter families'' of non-abelian entanglements, or in other words on the case where the associated modular curve has genus zero.

To state our main theorem precisely, we need to recall a few fundamentals about modular curves.
For an arbitrary level $m \in \mbn$, we let $X(m)$ denote the complete modular curve of level $m$, which parametrizes pairs $(E,\mc{B})$, where $E$ is an elliptic curve and $\mc{B} \subseteq E[m]$ is an ordered $\mbz/m\mbz$-basis of $E[m]$.  The curve $X(m)$ is equipped with a natural ``forgetful map''
\[
j_m : X(m) \longrightarrow X(1) \simeq \mbp^1,
\]
whose modular interpretation is $j_m\left( (E,\mc{B}) \right) = E$.
The group $\aut(E[m]) \simeq \GL_2(\mbz/m\mbz)$ acts on $X(m)$, and the kernel of this action is $\{ I, -I \} \subseteq \GL_2(\mbz/m\mbz)$.  For any $G(m) \subseteq \GL_2(\mbz/m\mbz)$, we enlarge $G(m)$ by setting
\[
\tilde{G}(m) := \langle G(m), -I \rangle
\]
and define the modular curve $X_{\tilde{G}(m)}$ to be the quotient curve of orbits under the action of $\tilde{G}(m)$:
\[
X_{\tilde{G}(m)} := X(m) / \tilde{G}(m).
\]
Let $F = \mbq(\mu_m)^{\det(G(m))}$ be the subfield of $\mbq(\mu_m)$ fixed by the subgroup 
\[
\det(G(m)) = \det(\tilde{G}(m)) \subseteq (\mbz/m\mbz)^\times \simeq \gal(\mbq(\mu_m)/\mbq). 
\]
The modular curve $X_{\tilde{G}(m)}$ is defined over $F$.  In particular, $X_{\tilde{G}(m)}$ is defined over $\mbq$ if and only if $\det(G(m)) = (\mbz/m\mbz)^\times$.   Furthermore, the forgetful map $j_m$ on $X(m)$ induces a map
\[
j_{\tilde{G}(m)} : X_{\tilde{G}(m)} \longrightarrow X(1) \simeq \mbp^1.
\]
Note that in the above construction, the modular curves $X_{\tilde{G}_1(m)}$ and $X_{\tilde{G}_2(m)}$ are isomorphic over $\mbq$ whenever $G_1(m)$ and $G_2(m)$ are conjugate inside $\GL_2(\mbz/m\mbz)$.  
It is thus sensible to coarsen the relations of equality and subset inclusion on the set of subgroups of $\GL_2(\mbz/m\mbz)$ to $\doteq$ and $\dot\subseteq$, where
\begin{equation} \label{defofdotrelations}
\begin{split}
G_1(m) \doteq G_2(m) \; &\myeq \; \exists g \in \GL_2(\mbz/m\mbz) \text{ with } G_1(m) = g G_2(m) g^{-1} \\
G_1(m) \, \dot\subseteq \, G_2(m) \; &\myeq \; \exists g \in \GL_2(\mbz/m\mbz) \text{ with } G_1(m) \subseteq g G_2(m) g^{-1}.
\end{split}
\end{equation}
Using this notation, 
a modular interpretation of rational points on $X_{\tilde{G}(m)}$ can be phrased as follows:  For any number field $K$ with $F \subseteq K$ and $x \in K - \{ 0, 1728 \}$, $x \in j_{\tilde{G}(m)}(X_{\tilde{G}(m)}(K))$ if and only if there is an elliptic curve $E$ defined over $K$ with $j$-invariant equal to $x$ and for which $\gal(K(E[m])/K) \, \dot\subseteq \, \tilde{G}(m)^t$.  In particular, we are choosing to let $\GL_2(\mbz/m\mbz)$ act on $X(m)$ on \emph{the left}\footnote{As is easily verified by direct computation, all subgroups $G(m) \subseteq \GL(\mbz/m\mbz)$ produced in the present paper satisfy the property that 
\[
G(m)^t := \{ g^t : g \in G(m) \}
\]
is $\GL_2(\mbz/m\mbz)$-conjugate to $G(m)$, and so our results are not affected by the choice of left action versus right action.}.  For a helpful discussion about this issue, see \cite[Remark 2.2]{rousezureickbrown}.  For full background details, see \cite{delignerapoport}.
 
 
 There is an infinite set of modular curves (see $\mc{G}_{\nonab}^{\max}$ below) whose $K$ rational points capture all elliptic curves over $K$ without complex multiplication that have a non-abelian entanglement.  Our main theorem explicitly describes the (finite) subset consisting of those modular curves having genus zero.
 Because the level $m$ will vary, we rephrase our definitions in terms of finite index (i.e. open) subgroups $G \subseteq \GL_2(\hat{\mbz})$, where
\[
\GL_2(\hat{\mbz}) = \lim_{\leftarrow} \GL_2(\mbz/m\mbz) \simeq \prod_p \GL_2(\mbz_p).
\]
For any open subgroup $G \subseteq \GL_2(\hat{\mbz})$, we denote by $m_G$ its \emph{level}, i.e. the smallest $m \in \mbn$ for which $\ker\left(\GL_2(\hat{\mbz}) \rightarrow \GL_2(\mbz/m\mbz) \right) \subseteq G$, and for any $m \in \mbn$ we define $G(m) := G \mod m \subseteq \GL_2(\mbz/m\mbz)$.  We extend our notation for the associated modular curve by setting 
\begin{equation} \label{defofGtilde}
\tilde{G} := \langle G, -I \rangle
\end{equation}
and setting the notation
\[
X_{\tilde{G}} := X_{\tilde{G}(m_{\tilde{G}})}, \quad\quad
j_{\tilde{G}} := j_{\tilde{G}(m_G)} : X_{\tilde{G}} \longrightarrow X(1).
\]

\begin{Definition} \label{defofnonabelianentanglementgroup}
Let $G \subseteq \GL_2(\hat{\mbz})$ be an open subgroup of level $m_G$.  We say that $G$ is a {\bf{non-abelian entanglement group}} if there is a number field $K$ and an elliptic curve $E$ over $K$ having a non-abelian entanglement at level $m_G$ and satisfying $G(m_G) \doteq \gal(K(E[m_G])/K)$.  We call $G$ a {\bf{maximal non-abelian entanglement group}} if $G$ is a non-abelian entanglement group that is maximal with respect to $\dot\subseteq$ among all non-abelian entanglement groups.
\end{Definition}
\begin{remark}
One can of course define the notion of a non-abelian entanglement group in purely group-theoretical terms (see Remark \ref{purelygrouptheoreticremark}). 
\end{remark}
Next we elaborate on a technicality that arises from the distinction between $G$ and $\tilde{G}$ in the case when $-I \notin G$.  For a given elliptic curve $E$ over $K$ and open subgroup $G \subseteq \GL_2(\hat{\mbz})$, the property of whether or not $\gal(K(E[m])/K) \, \dot\subseteq \, \tilde{G}(m)$ is independent of twisting in the sense that it is a function just of the $j$-invariant of $E$ (i.e. of the $\ol{K}$-isomorphism class of $E$).  By contrast, in case $-I \notin G(m)$, the finer question of whether or not $\gal(K(E[m])/K) \, \dot\subseteq \, G(m)$ for $E$ corresponding to a point of $X_{\tilde{G}}(K)$ may change as we twist $E$ (i.e. as we vary $E$ within a fixed $\ol{K}$-isomorphism class).  This motivates the following terminology.
\begin{Definition}
We say that a subgroup $G \subseteq \GL_2(\hat{\mbz})$ is {\bf{twist-independent}} if $-I \in G$; otherwise we say that $G$ is {\bf{twist-dependent}}.
\end{Definition}
We now fix notation used in the main theorem.  Define the following subgroups $G_m \subseteq \GL_2(\hat{\mbz})$:
\begin{equation} \label{defofGs}
\begin{split}
G_6 &:= \left\{ g \in \GL_2(\hat{\mbz}) : g \mod 6 \in \left\langle \begin{pmatrix} 1 & 1 \\ 0 & 5 \end{pmatrix}, \begin{pmatrix} 5 & 1 \\ 3 & 2 \end{pmatrix}, \begin{pmatrix} 5 & 4 \\ 4 & 1 \end{pmatrix} \right\rangle \right\}, \\
G_{10} &:=  \left\{ g \in \GL_2(\hat{\mbz}) : g \mod 10 \in \left\langle \begin{pmatrix} 5 & 6 \\ 4 & 5 \end{pmatrix}, \begin{pmatrix} 4 & 9 \\ 9 & 6 \end{pmatrix}, \begin{pmatrix} 7 & 3 \\ 9 & 4 \end{pmatrix} \right\rangle \right\}, \\
G_{15} &:=  \left\{ g \in \GL_2(\hat{\mbz}) : g \mod 15 \in \left\langle \begin{pmatrix} 2 & 3 \\ 14 & 14 \end{pmatrix}, \begin{pmatrix} 4 & 0 \\ 0 & 1 \end{pmatrix}, \begin{pmatrix} 0 & 2 \\ 14 & 0 \end{pmatrix} \right\rangle \right\}, \\
G_{18} &:=  \left\{ g \in \GL_2(\hat{\mbz}) : g \mod 18 \in \left\langle \begin{pmatrix} 7 & 17 \\ 0 & 5 \end{pmatrix}, \begin{pmatrix} 17 & 3 \\ 3 & 14 \end{pmatrix}, \begin{pmatrix} 4 & 3 \\ 3 & 14 \end{pmatrix} \right\rangle \right\}.
\end{split}
\end{equation}
Note that each of the underlying levels is divisible by exactly 2 primes, and thus we have the unique permissible factorizations
\begin{equation} \label{listoffactorizations}
6 = 2\cdot 3, \quad\quad 10 = 2\cdot 5, \quad\quad 15 = 3\cdot 5, \quad\quad 18 = 2\cdot 9.
\end{equation}
Also, each of these groups is checked to be twist-independent.  Next, we define the rational functions $j_6$, $j_{10}$, $j_{15}$ and $j_{18}$ by
\begin{equation} \label{defofj6}
\begin{split}
j_{6}(t) := &2^{10}\,3^3\, t^3 (1 - 4t^3) \\
j_{10}(t) := & s_{10}^3(s_{10}^2+5s_{10}+40),\qquad s_{10} =\frac{3t^6+12t^5+80t^4+50t^3-20t^2-8t+8}{(t-1)^2(t^2+3t+1)^2} \\
j_{15}(t) := & s_{15}^3(s_{15}^2+5s_{15}+40),\qquad s_{15}=t^3-\frac{5-3\sqrt{-15}}{2}\\
j_{18}(t) := & \frac{-3^3\,t^3(t^3-2)(3t^3-4)^3(3t^3-2)^3}{(t^3-1)^2}.
\end {split}
\end{equation}
For $g \in \mbz_{\geq 0}$ we set
\[
\begin{split}
\mc{G} &:= \{G \subseteq \GL_2(\hat{\mbz}) : \text{ $G$ is open} \}, \quad\quad\quad\quad\quad\quad\quad\quad\quad\quad\quad\quad\quad\, \mc{G}(g) := \{ G \in \mc{G} : \genus(X_{\tilde{G}}) = g \}, \\
\mc{G}_{\nonab} &:= \{ G \in \mc{G} : G \text{ is a non-abelian entanglement group} \}, \quad \mc{G}_{\nonab}(g) := \mc{G}_{\nonab} \cap \mc{G}(g), \\
\mc{G}_{\nonab}^{\max} &:= \{G \in \mc{G}_{\nonab}: \nexists H \in \mc{G}_{\nonab} \text{ with } G \, \dot\subsetneq \, H \}, \quad\quad\quad\quad\;\; \mc{G}_{\nonab}^{\max}(g) := \mc{G}_{\nonab}^{\max} \cap \mc{G}(g),
\end{split}
\]
where we are extending the definitions \eqref{defofdotrelations} of $\doteq$ and $\dot\subseteq$ in the obvious way to subgroups of $\GL_2(\hat{\mbz})$, and $G \, \dot\subsetneq \, H$ means that $G \, \dot\subseteq \, H$ and $G \not\doteq H$.
Furthermore, we extend the relations $\dot\subseteq$ and  $\doteq$ to subsets $\mc{S}_1, \mc{S}_2 \subseteq \mc{G}$ by declaring that
\[
\begin{split}
\mc{S}_1 \, \dot\subseteq \, \mc{S}_2 \; &\myeq \; \forall G_1 \in \mc{S}_1\, \exists G_2 \in \mc{S}_2 \text{ with } G_1 \doteq G_2 \\
\mc{S}_1 \doteq \mc{S}_2 \; &\myeq \; \mc{S}_1 \, \dot\subseteq \, \mc{S}_2 \text{ and } \mc{S}_2 \, \dot\subseteq \, \mc{S}_1.
\end{split}
\]
In particular, note that one could have $\mc{S}_1 \doteq \mc{S}_2$ even though $\# \mc{S}_1 \neq \#\mc{S}_2$, since for any single element $G_1 \in \mc{S}_1$, we could have $G_1 \doteq G_2$ for many different $G_2 \in \mc{S}_2$.
\begin{Theorem} \label{mainthm}
We have 
\begin{equation} \label{listofmaximalgroups}
\mc{G}_{\nonab}^{\maxm}(0) \doteq \{ G_6, G_{10}, G_{15}, G_{18} \},
\end{equation}
where the groups $G_m$ are as in \eqref{defofGs}.  In other words, every group $G \in \mc{G}_{\nonab}^{\maxm}(0)$ is $\GL_2(\hat{\mbz})$-conjugate to exactly one of the groups $G_m$ appearing in the right-hand set.
Furthermore, each group $G_m$ is twist-independent of level $m$, and there is a parameter $t$ on $X_{G_m}$ for which  
\[
j_{G_m}(t) = j_m(t),
\]
where $j_m(t)$ is as in \eqref{defofj6}.  The modular curves $X_{G_6}$, $X_{G_{10}}$, and $X_{G_{18}}$ are defined over $\mbq$, whereas the modular curve $X_{G_{15}}$ is defined over $\mbq(\sqrt{-15})$.  Finally, in all cases the underlying entanglement is an $S_3$-entanglement, i.e. for each $G_m \in \mc{G}_{\nonab}^{\maxm}(0)$ and for each elliptic curve $E$ over a number field $K$ satisfying $j(E) \in j_{G_m}(X_{G_m}(K))$ and $\gal(K(E[m])/K) \doteq G_m(m)$, we have 
\[
\gal(K(E[m_1]) \cap K(E[m_2]) / K) \simeq S_3,
\]
where $m = m_1m_2$ is the unique permissible factorization of $m$ as in \eqref{listoffactorizations} and $S_3$ denotes the symmetric group of order 6.
\end{Theorem}

Theorem \ref{mainthm} may be restated in terms of elliptic curves over $K(t)$ as follows.
\begin{Theorem} \label{corollarythm}
Let $K$ be a number field and let $E$ be an elliptic curve defined over $K(t)$.  Then $E$ has a non-abelian entanglement over $K(t)$ if and only if the $j$-invariant $j_E(t) \in K(t)$ satisfies
\[
j_E(t) \in \{ j_6(f(t)), j_{10}(f(t)), j_{15}(f(t)), j_{18}(f(t)) : f(t) \in K(t) \},
\]
where the rational functions $j_6$, $j_{10}$, $j_{15}$ and $j_{18}$ are as in \eqref{defofj6}.  The case $j_E(t) = j_{15}(f(t))$ can only happen if $\sqrt{-15} \in K$.  Finally, if $j_E(t) = j_m(f(t))$ for some $f(t) \in K(t)$, then $E$ has a non-abelian entanglement at level $m$ and the underlying entanglement has Galois group $S_3$ over $K(t)$.
\end{Theorem}

\begin{remark}
The infinite family of $j$-invariants $j_6(t) = 2^{10}\,3^3\, t^3 (1 - 4t^3)$ was considered in previous work of the first author (see \cite{braujones}).  In that paper, it is incorrectly stated that, for any elliptic curve $E$ over $\mbq$ with $j$-invariant $j_E$, we have $j_E = j_6(t_0)$ for some $t_0 \in \mbq$ if and only if $E \simeq_{\ol{\mbq}} E'$ for some elliptic curve $E'$ over $\mbq$ satisfying $\mbq(E'[2]) \subseteq \mbq(E'[3])$.  Although the ``only if'' part is correct, the converse can fail for elliptic curves $E/\mbq$ satisfying $\mbq(E[2]) = \mbq$.  A correct biconditional statement is as follows:  For each elliptic curve $E$ over $\mbq$ with $j$-invariant $j_E \in \mbq - \{0, 1728 \}$, 
\[
j_E = j_6(t_0) \text{ for some } t_0 \in \mbq \; \Longleftrightarrow \;  \exists \, E' / \mbq \text{ with } E' \simeq_{\ol{\mbq}} E, \,  [\mbq(E'[2]) : \mbq] = 6 \text{ and } \mbq(E'[2]) \subseteq \mbq(E'[3]).
\]
The first author thanks Maarten Derickx for pointing this out.
\end{remark}

\begin{remark}
When $K = \mbq$, Theorem \ref{mainthm} leads in some cases to precise criteria for detecting elliptic curves over $\mbq$ for which every $\gal(\mbq(E[n])/\mbq)$ is as large as possible, relative to a given obstruction.  We discuss this in more detail in Section \ref{Section:Applications}. Another motivation to consider Problem \ref{entanglementproblem} is its relationship to constants decorating the main term in various conjectures attached to elliptic curves (see \cite{brauthesis}).
\end{remark}


The proof of Theorem \ref{mainthm} breaks up into two main steps. The first is to establish Proposition \ref{keypropositionformainthm} below, which reduces the problem to a finite search and hence enables us to verify \eqref{listofmaximalgroups} by explicit computation. The proposition is established in Section \ref{Section:FiniteSearch} by a series of technical group-theoretical lemmas, essentially deriving properties of $G$ that are visible at the lower $\SL_2$-level whenever the $\GL_2$-level and $\SL_2$-level differ (see Definition \ref{defofGL2levelandSL2level} below). For $g\geq 1$, the latter statement of Proposition \ref{keypropositionformainthm} is false, in that even {\em maximal} non-abelian entanglement groups can have distinct $\SL_2$-level and $\GL_2$-level. The proposition also fails to hold, even for $g=0$, if we remove the maximality assumption. To illustrate this fact, we have included in Section \ref{infiniteD6family} an infinite family of (non-maximal) genus $0$ non-abelian entanglement groups with unbounded $\GL_2$-level. The second main step in the proof of Theorem \ref{mainthm} is to derive explicit models for the modular curves, as well as the corresponding maps to the $j$-line. This work is done in Section \ref{Section:ExplicitModels}.

For any open subgroup $G \subseteq \GL_2(\hat{\mbz})$, we recall and extend the concept of its level $m_G$ in the following definition.
\begin{Definition} \label{defofGL2levelandSL2level}
For an open subgroup $G \subseteq \GL_2(\hat{\mbz})$, we define the positive integer
\[
m_{\GL_2}(G) := \min \left\{ m \in \mbn : \ker\left( \GL_2(\hat{\mbz}) \rightarrow \GL_2(\mbz/m\mbz) \right) \subseteq G \right\}
\]
and call it the {\textbf{$\GL_2$-level of $G$}}.  Furthermore, we define the {\textbf{$\SL_2$-level of $G$}} by
\[
m_{\SL_2}(G) := \min \left\{ m \in \mbn : \ker\left( \SL_2(\hat{\mbz}) \rightarrow \SL_2(\mbz/m\mbz) \right) \subseteq G \right\}.
\]
\end{Definition}
\noindent
It is straightforward to see that $m_{\SL_2}(G)$ always divides $m_{\GL_2}(G)$; they may or may not be equal.  Next, for any level $m \in \mbn$, we define
\[
\begin{split}
\mc{G}_{\nonab}^{m_{\SL_2} = m} :=& \; \{ G \in \mc{G}_{\nonab} : \, m_{\SL_2}(G) = m \}, \quad\quad\; \mc{G}_{\nonab}^{m_{\SL_2} = m}(g) := \mc{G}_{\nonab}^{m_{\SL_2} = m} \cap \mc{G}(g), \\
\mc{G}_{\nonab}^{m_{\GL_2} = m} :=& \;  \{ G \in \mc{G}_{\nonab} : \, m_{\GL_2}(G) = m \}, \quad\quad \mc{G}_{\nonab}^{m_{\GL_2} = m}(g) := \mc{G}_{\nonab}^{m_{\GL_2} = m} \cap \mc{G}(g).
\end{split}
\]
\begin{proposition} \label{keypropositionformainthm}
With the notation just outlined, we have
\begin{equation} \label{finitelistforkeyprop}
\mc{G}_{\nonab}(0) = \bigsqcup_{m \in \mc{L}} \mc{G}_{\nonab}^{m_{\SL_2} = m}(0),
\end{equation}
where $\mc{L} = \{ 6, 10, 12, 15, 18, 20, 24, 30, 36, 40, 48, 60, 72, 96 \}$.  Furthermore, for every $G \in \mc{G}_{\nonab}^{\max}(0)$, we have $m_{\GL_2}(G) = m_{\SL_2}(G)$.
\end{proposition}

As a byproduct of the computations involved in the proof of Proposition \ref{keypropositionformainthm}, we obtain, for each $m \in \mc{L}$, an explicit list of the groups $G \in \mc{G}_{\nonab}^{m_{\GL_2} = m}(0) / \doteq$.
Table \ref{tableforlevelsneq30} lists, for each $m \in \mc{L} - \{ 30, 60 \}$, the number of groups in $\mc{G}_{\nonab}^{m_{\GL_2} = m}(0) / \doteq$.  Furthermore, it details which entanglement groups occur.  More precisely, for any non-abelian entanglement group $G$ of level $m$ and elliptic curve $E$ over $K$ satisfying
\begin{equation} \label{GisGaloisGroup}
G(m) \doteq \gal(K(E[m])/K),
\end{equation}
there exists by definition a permissible factorization $m = m_1m_2$ so that
\begin{equation} \label{defofHquotient}
H := \gal(K(E[m_1] \cap K(E[m_2]) / K) 
\end{equation} 
is a non-abelian group.  By \eqref{GisGaloisGroup} and the Galois correspondence, the group $H$ in \eqref{defofHquotient} is uniquely determined by $G$ and the pair $(m_1,m_2)$; we call $H$ \emph{the quotient associated to $G$ and $(m_1,m_2)$}.  
In case the level $m$ has only two distinct primes in its factorization, the co-prime integers $m_1$ and $m_2$ satisfying $m = m_1m_2$ are uniquely determined; in this case we simply call $H$ \emph{the quotient associated to $G$} and define
\begin{equation} \label{defofHofm}
\mc{G}_{\nonab}^{m_{\GL_2} = m}(g,H) := \{ G \in \mc{G}_{\nonab}^{m_{\GL_2} = m}(g) : \, H \text{ is the quotient associated to $G$} \}.
\end{equation}
There are exactly 3 groups $H$ that arise as non-abelian quotients associated to $\ds G \in \bigcup_{m \in \mc{L}} \mc{G}_{\nonab}^{m_{\GL_2} = m}(0)$, namely the dihedral groups $D_3$ ($\simeq S_3$) and $D_6$ of orders 6 and 12 respectively, and the dicyclic group\footnote{The dicyclic group satisfies $\dic_3 \simeq \mbz/4\mbz \ltimes \mbz/3\mbz$, where the map $\mbz/4\mbz \rightarrow \aut(\mbz/3\mbz)$ defining the semidirect product structure is the unique non-trivial group homomorphism.} $\dic_3$ of order 12.  For levels $m \in \mc{L}$ that are divisible by just two distinct primes, our results give the data displayed in Table \ref{tableforlevelsneq30}.

\begin{table}[!ht]
\[
\begin{array}{|c||c|c|c|c|} \hline
 & & & & \\[-.75 em]
m & 
\left| \, \mc{G}_{\nonab}^{m_{\GL_2}=m}\left( 0 \right)/\doteq \, \right| & 
\left| \,  \mc{G}_{\nonab}^{m_{\GL_2}=m}\left( 0,D_3 \right) / \doteq \, \right| & 
\left| \,  \mc{G}_{\nonab}^{m_{\GL_2}=m}\left( 0,D_6 \right) / \doteq \, \right| & 
\left| \,  \mc{G}_{\nonab}^{m_{\GL_2}=m}\left( 0,\dic_3 \right) / \doteq \, \right| \\
 & & & & \\[-.75 em]
\hline \hline 6 & 4 & 4 & 0 & 0 \\ 
\hline 10 & 1 & 1 & 0 & 0 \\ 
\hline 12 & 12 & 10 & 2 & 0 \\ 
\hline 15 & 1 & 1 & 0 & 0 \\ 
\hline 18 & 10 & 10 & 0 & 0 \\
\hline 20 & 4 & 3 & 0 & 1 \\
\hline 24 & 54 & 38 & 16 & 0 \\
\hline 36 & 30 & 24 & 6 & 0 \\    
\hline 40 & 2 & 1 & 0 & 1 \\
\hline 48 & 56 & 40 & 16 & 0 \\
\hline 72 & 38 & 6 & 32 & 0 \\
\hline 96 & 12 & 4 & 8 & 0 \\
\hline
\end{array}
\]
\vspace{.1in}
\caption{Frequencies of genus zero non-abelian entanglement groups of level $\in \mc{L} \backslash \{ 30, 60 \}$}
\label{tableforlevelsneq30}
\end{table}

For the remaining levels $m \in \{ 30, 60 \}$, we must refine \eqref{defofHofm} to reflect the dependence on the pair $(m_1,m_2)$ occurring in the permissible factorization $m = m_1m_2$, which isn't unique in this case.  We define
\[
\begin{split}
\mc{G}_{\nonab}^{m_{\GL_2} = m}\left( g,(m_1,m_2) \right) &:= \{ G \in \mc{G}_{\nonab}^{m_{\GL_2} = m}(g) : \text{ the quotient associated to $G$ and $(m_1,m_2)$ is non-abelian} \}, \\
\mc{G}_{\nonab}^{m_{\GL_2} = m}\left( g,(m_1,m_2),H \right) &:= \{ G \in \mc{G}_{\nonab}^{m_{\GL_2} = m}\left( g,(m_1,m_2) \right) : \, H \text{ is the quotient associated to $G$ and $(m_1,m_2)$} \}.
\end{split}
\]
Regarding $m \in \{ 30, 60 \}$, the non-abelian group $H$ above is found to be either $D_3$ or $D_6$, and these groups occur with the frequencies indicated in Table \ref{tableforlevel30} and Table \ref{tableforlevel60}.
\begin{table}[!ht]
\[
\begin{array}{|c||c|c|c|} \hline 
 & & & \\[-.75 em]
 (m_1,m_2) &  
 \left| \, \mc{G}_{\nonab}^{m_{\GL_2} = 30}\left( 0,(m_1,m_2) \right) / \doteq \, \right| & 
 \left| \, \mc{G}_{\nonab}^{m_{\GL_2} = 30}\left( 0,(m_1,m_2), D_3 \right) / \doteq \, \right| & 
 \left| \, \mc{G}_{\nonab}^{m_{\GL_2} = 30}\left( 0,(m_1,m_2), D_6 \right) / \doteq \, \right| \\ 
 & & & \\[-.75 em]
\hline \hline (2,15) & 22 & 22 & 0 \\ 
\hline (3,10) & 20 & 16 & 4 \\ 
\hline (5,6) & 2 & 2 & 0  \\
\hline
\end{array}
\]
\vspace{.1in}
\caption{Frequencies of genus zero non-abelian entanglement groups of level $= 30$}
\label{tableforlevel30}
\end{table}

\begin{table}[!ht]
\[
\begin{array}{|c||c|c|c|} \hline 
& & & \\[-.75 em]
(m_1,m_2) &  
\left| \, \mc{G}_{\nonab}^{m_{\GL_2} = 60}\left( 0,(m_1,m_2) \right) / \doteq \, \right| & 
\left| \, \mc{G}_{\nonab}^{m_{\GL_2} = 60}\left( 0,(m_1,m_2), D_3 \right) / \doteq \, \right| & 
\left| \, \mc{G}_{\nonab}^{m_{\GL_2} = 60}\left( 0,(m_1,m_2), D_6 \right) / \doteq \, \right| \\ 
 & & & \\[-.75 em]
\hline \hline (4,15) & 14 & 0 & 14 \\ 
\hline (3,20) & 14 & 0 & 14 \\ 
\hline (5,12) & 0 & 0 & 0  \\
\hline
\end{array}
\]
\vspace{.1in}
\caption{Frequencies of genus zero non-abelian entanglement groups of level $= 60$}
\label{tableforlevel60}
\end{table}

What can we say about groups $G \in \mc{G}_{\nonab}(0)$ satisfying $m_{\GL_2(G)} > 96$?  As mentioned earlier, we will see that the set
\[
\{ G \in \mc{G}_{\nonab}(0) : m_{\SL_2}(G) = m, \; m_{\GL_2}(G) > m \}
\]
is infinite for some $m \in \mc{L}$ (see Section \ref{infiniteD6family}).  We emphasize that, according to Proposition \ref{keypropositionformainthm}, this does not happen when we restrict to $\mc{G}_{\nonab}^{\max}(0)$, i.e. we have
\[
\{ G \in \mc{G}_{\nonab}^{\max}(0) : m_{\SL_2}(G) < \; m_{\GL_2}(G) \} = \emptyset.
\]

\subsection{Acknowledgements}

The authors would like to thank David Zureick-Brown for insightful conversations and also Jackson Morrow and Harris Daniels for helpful comments on an earlier version of the paper.

\bigskip

\section{An application to counting elliptic curves over $\mbq$ with maximal Galois image modulo a prescribed obstruction} \label{Section:Applications}

In this section we discuss an application of Theorem \ref{mainthm} to the problem of determining which elliptic curves defined over $\mbq$ have Galois image as large as possible relative to a given obstruction, and also of counting elliptic curves with this property.  More precisely, here and throughout the paper, let $E_{\tors} := \bigcup_{m = 1}^{\infty} E[m]$ denote the torsion subgroup of $E$ over $\ol{\mbq}$, let $G_\mbq := \gal(\ol{\mbq}/\mbq)$ denote the absolute Galois group of $\mbq$ and let
\[
\begin{split}
\rho_E : G_{\mbq} &\longrightarrow \aut(E_{\tors}) \simeq \GL_2(\hat{\mbz}), \\
\rho_{E,m} : G_{\mbq} &\longrightarrow \aut(E[m]) \simeq \GL_2(\mbz/m\mbz)
\end{split}
\]
be the Galois representations defined by letting $G_{\mbq}$ act on $E_{\tors}$ (resp. on $E[m]$) and fixing a $\hat{\mbz}$-basis (resp. a $\mbz/m\mbz$-basis) thereof.  Furthermore, let 
$G \subseteq \GL_2(\hat{\mbz})$ be an open subgroup and suppose that there is an elliptic curve $E$ over $\mbq$ satisfying $\rho_E(G_\mbq) \, \dot\subseteq \, G$.  In fact, this will imply that $G$ is \emph{admissible} in the sense of the following definition.
\begin{Definition} \label{defofadmissible}
An open subgroup $G \subseteq \GL_2(\hat{\mbz})$ is called \textbf{admissible} if 
\begin{enumerate}
\item $\det G = \hat{\mbz}^\times$, and
\item $\exists g \in G$ that is $\GL_2(\hat{\mbz})$-conjugate either to $\begin{pmatrix} 1 & 0 \\ 0 & -1 \end{pmatrix}$ or to $\begin{pmatrix} 1 & 1 \\ 0 & -1 \end{pmatrix}$.
\end{enumerate}
\end{Definition}
By considering the Weil pairing and the image under $\rho_E$ of a complex conjugation, we may see that
\[
\exists \text{ an elliptic curve } E/\mbq \text{ with } \rho_E(G_\mbq) \, \dot\subseteq \, G \; \Longrightarrow \; G \text{ is admissible.}
\]
\begin{remark}
Our restriction to considering only elliptic curves defined over $\mbq$ applies only to this section of the paper, and not to other sections. In particular, the computer search associated to Theorem \ref{mainthm} is not restricted to admissible subgroups of $\GL_2(\hat{\mbz})$, and indeed the group $G_{15}$ of \eqref{defofGs} is not admissible, failing each of the conditions in Definition \ref{defofadmissible}.
\end{remark}
\begin{remark}
Definition \ref{defofadmissible} differs slightly from the definition of admissible found in \cite[p. 8]{sutherlandzywina}, wherein it is also demanded that $-I \in G$ and that $G$ be of prime power level. As a consequence of the Hasse-Minkowski theorem, assuming $-I \in G$ and $X_G$ has genus zero, they prove that
\[
\text{ $G$ is admissible and of prime power level} \; \Longrightarrow \; \left| X_G(\mbq) \right| = \infty. 
\]
\end{remark}
Given an admissible open subgroup $G \subseteq \GL_2(\hat{\mbz})$, it is natural to wonder whether or not there exists an elliptic curve $E$ over $\mbq$ satisfying
\begin{equation} \label{canrhoequalG}
\rho_{E}(G_\mbq) \doteq G.
\end{equation}
Because we are working over $\mbq$, classical class field theory motivates the following definition.
\begin{Definition}
We say that a subgroup $G \subseteq \GL_2(\hat{\mbz})$ is \textbf{commutator-thick} if
\[
\left[ G, G \right] = G \cap \SL_2(\hat{\mbz}).
\]
\end{Definition}
Note that, for any subgroup $G \subseteq \GL_2(\hat{\mbz})$, we clearly have $\left[ G, G \right] \subseteq G \cap \SL_2(\hat{\mbz})$, and this containment can be proper (indeed, it is proper even for $G = \GL_2(\hat{\mbz})$)\footnote{Here we are defining the commutator subgroup $[G,G]$ to be the \emph{closure} of the subgroup generated by commutators.}.  Furthermore, it follows from the Kronecker-Weber theorem that, for any elliptic curve $E$ over $\mbq$, the subgroup $\rho_E(G_\mbq) \subseteq \GL_2(\hat{\mbz})$ is commutator-thick.  Indeed, identifying $\rho_E(G_\mbq)$ with $\gal(\mbq(E_{\tors})/\mbq)$, we have
\begin{equation} \label{whyrhosubEofGQiscommutatorthick}
\mbq(\mu_\infty) = \mbq(E_{\tors})^{\rho_E(G_\mbq) \cap \SL_2(\hat{\mbz})} \subseteq \mbq(E_{\tors})^{\left[ \rho_E(G_\mbq), \rho_E(G_\mbq) \right]} \subseteq \mbq^{\ab} = \mbq(\mu_\infty),
\end{equation}
and so we must have equality at each inclusion.  In particular, \eqref{canrhoequalG} can only happen if $G$ is itself  commutator-thick.  In case $G$ is not commutator-thick, we are motivated to consider what it should mean for $\rho_E(G_\mbq) \subseteq G$ to be ``as large as possible.''
Following \cite{jones2}, we make the following definition.
\begin{Definition} \label{commutatormaximaldef}
Let $G \subseteq \GL_2(\hat{\mbz})$.  Given a subgroup $H \subseteq G$, we say that $H$ is \textbf{commutator-maximal in $G$} if 
\[
\left[ H, H \right] = \left[ G, G \right].
\]
If $H \, \dot\subseteq \, G$, we call $H$ commutator-maximal in $G$ just in case $gHg^{-1}$ is commutator-maximal in $G$, for some (any) $g \in \GL_2(\hat{\mbz})$ for which $gHg^{-1} \subseteq G$.
\end{Definition}
We will use commutator-maximality of $H = \rho_E(G_\mbq) \, \dot\subseteq \, G$ to define the concept of $\rho_E(G_\mbq)$ having maximal image inside $G$.  Since we are assuming that $G \subseteq \GL_2(\hat{\mbz})$ is an open subgroup, it follows that $\left[ G, G \right] \subseteq \SL_2(\hat{\mbz})$ is open, which implies that the index of $\left[ G, G \right]$ in $G \cap \SL_2(\hat{\mbz})$ is finite.  As discussed in \cite{jones2}, in case $\det H = \hat{\mbz}^\times$, Definition \ref{commutatormaximaldef} is equivalent to the statement that
\begin{equation} \label{equalityofsomeindices}
\left[ H : G \right] = \left[ [G,G] : G \cap \SL_2(\hat{\mbz}) \right].
\end{equation}
In case $G = \GL_2(\hat{\mbz})$, index on the right-hand side of \eqref{equalityofsomeindices} is $2$; thus in this case $\rho_E(G_\mbq)$ is commutator-maximal in $\GL_2(\hat{\mbz})$ if and only if $\rho_E(G_\mbq)$ has index two inside $\GL_2(\hat{\mbz})$.  An elliptic curve $E$ for which $\left[ \rho_E(G_\mbq) : \GL_2(\hat{\mbz}) \right] = 2$ is typically called a \emph{Serre curve}, and so this motivates the following nomenclature.  
\begin{Definition} \label{defofGSerrecurve}
Let $G \subseteq \GL_2(\hat{\mbz})$ be an admissible open subgroup and suppose that $E$ is an elliptic curve over $\mbq$ that satisfies $\rho_E(G_\mbq) \, \dot\subseteq \, G$.  We call $E$ a \textbf{$G$-Serre curve} if $\rho_E(G_\mbq)$ is commutator-maximal in $G$, in the sense of Definition \ref{commutatormaximaldef}.
\end{Definition} 
\begin{remark}
Definition \ref{defofGSerrecurve} is \emph{stronger} than (and in particular not equivalent to) the condition that $\rho_E(G_{\mbq})$ be maximal among commutator-thick subgroups. For example, there exist elliptic curves $E$ over $\mbq$ for which 
\begin{equation} \label{equalsmbqi}
    \mbq(\sqrt{\gD_E}) = \mbq(i)
\end{equation} 
and with $\rho_E(G_\mbq) \subseteq \GL_2(\hat{\mbz})$ maximal among commutator-thick subgroups, but, since the index two subgroup of $\GL_2(\hat{\mbz})$ corresponding to \eqref{equalsmbqi} is not commutator-thick, none of these elliptic curves will be Serre curves.  (In this case, $\rho_E(G_\mbq)$ must be contained in an index four subgroup of $\GL_2(\hat{\mbz})$.)
\end{remark}
If $E$ is an elliptic curve over $\mbq$ with $\rho_E(G_\mbq) \, \dot\subseteq \, G$, how can we tell whether or not $E$ is a $G$-Serre curve?  We define the following two sets of proper subgroups of $G$:
\begin{equation} \label{defsofSofGandSsupmaxofG}
\begin{split}
\mf{S}(G) :=& \left\{ H \subsetneq G : \, H \text{ is admissible but not commutator-maximal in $G$} \right\}, \\
\mf{S}^{\max}(G) :=& \left\{ H \in \mf{S}(G) : \, \nexists H_1 \in \mf{S}(G) \text{ for which } H \subsetneq H_1 \subsetneq G \right\}.
\end{split}
\end{equation}
Since we obviously have
\begin{equation} \label{ifEisnotaGSerrecurveequivalence}
\text{$E$ is not a $G$-Serre curve} \; \Longleftrightarrow \; \exists H \in \mf{S}^{\max}(G) \text{ for which } \rho_E(G_\mbq) \, \dot\subseteq \, H,
\end{equation}
it is of natural interest to determine the set $\mf{S}^{\max}(G)$.  The following theorem does so, for a particular open subgroup $G \subseteq \GL_2(\hat{\mbz})$.  As we will see, for each $H \in \{ G_6, G_{18} \}$ (the two non-abelian entanglement groups featured in Theorem \ref{mainthm}), we have $H \cap G \in \mf{S}^{\max}(G)$.  To begin with, we will observe that this is not unexpected, since whenever $H \subseteq G$ is a fibered product over a non-abelian quotient and $G$ is merely a fibered product over a cyclic quotient, then $H$ is not commutator-maximal in $G$.  In particular, consider the following two lemmas, where $G_1$ and $G_2$ are finite groups and $\psi_i : G_i \longrightarrow \Gamma_\psi$ are surjective group homomorphisms onto a common quotient group $\Gamma_\psi$.  We let 
\[
G_1 \times_\psi G_2 := \{ (g_1,g_2) \in G_1 \times G_2 : \psi_1(g_1) = \psi_2(g_2) \}
\]
denote the fibered product group.  We note that non-abelian entanglement groups may be defined in terms of fibered products.
\begin{remark} \label{purelygrouptheoreticremark}
By considering Definition \ref{defofnonabelianentanglementgroup} and the Galois correspondence, we may see that an open subgroup $G \subseteq \GL_2(\hat{\mbz})$ is a non-abelian entanglement group if and only if there is a level $m \in \mbn$ that admits a permissible factorization $m = m_1 m_2$ and, under the isomorphism $\GL_2(\mbz/m\mbz) \simeq \GL_2(\mbz/m_1\mbz) \times \GL_2(\mbz/m_2\mbz)$ of the Chinese remainder theorem, we have $G(m) \simeq G(m_1) \times_{\psi} G(m_2)$, where the associated common quotient $\Gamma_\psi$ a non-abelian group.
\end{remark}
\begin{lemma} \label{dissolvesundercommutatorslemma}
With the notation as above, if the group $\Gamma_\psi$ is cyclic, then we have
\[
\left[ G_1 \times_\psi G_2, G_1 \times_\psi G_2 \right] = \left[ G_1, G_1 \right] \times \left[ G_2, G_2 \right].
\]
\end{lemma}
\begin{proof}
This follows from \cite[Lemma 1, p. 174]{langtrotter}.
\end{proof}
By contrast, we have
\begin{lemma} \label{doesntdissolvelemma}
With the notation as above, if the group $\Gamma_\psi$ is non-abelian, then
\[
\left[ G_1 \times_\psi G_2, G_1 \times_\psi G_2 \right] \subsetneq  \left[ G_1, G_1 \right] \times \left[ G_2, G_2 \right].
\]
\end{lemma}
\begin{proof}
Since $\Gamma_\psi$ is non-abelian, we have $[\Gamma_\psi,\Gamma_\psi] \neq \{ 1 \}$.  Since each $\psi_i$ is onto, we have
\[
\{ 1 \} \neq [\Gamma_\psi, \Gamma_\psi] = \psi_i \left( [G_i, G_i] \right) \subseteq \Gamma_\psi \quad\quad \left( i \in \{1, 2 \} \right),
\]
and so the commutator subgroup 
\[
\left[ G_1 \times_\psi G_2, G_1 \times_\psi G_2 \right] \subseteq \left[ G_1, G_1 \right] \times_{\psi} \left[ G_2, G_2 \right] 
\]
is contained in a fibered product over the non-trivial group $[\Gamma_\psi,\Gamma_\psi]$, and is therefore a proper subgroup of $\left[ G_1, G_1 \right] \times \left[ G_2, G_2 \right]$.
\end{proof}
Combining Lemma \ref{dissolvesundercommutatorslemma} with Lemma \ref{doesntdissolvelemma}, we obtain the following corollary.
\begin{cor} \label{nonabentanglementcorollary}
Let $G = G_1 \times_{\psi} G_2$ be a fibered product over a \emph{cyclic} group $\Gamma_\psi$ and, for each $i \in \{ 1, 2 \}$, let $H_i \subseteq G_i$ be a subgroup.  Suppose that $H \subseteq G$ is a subgroup of the form $H = H_1 \times_{\phi} H_2$, where each $\phi_i : H_i \rightarrow \Gamma_\phi$ is a surjective homomorphism onto a \emph{non-abelian} group $\Gamma_\phi$.  Then $H$ is \emph{not} commutator-maximal in $G$.
\end{cor}
We now take $G_1 := \GL_2(\mbz/2\mbz)$ and $\ds G_2 := \left\{ \begin{pmatrix} * & * \\ 0 & * \end{pmatrix} \right\} \subseteq \GL_2(\mbz/3\mbz)$.  We let $\gamma := \begin{pmatrix} 1 & 1 \\ 1 & 0 \end{pmatrix} \in \GL_2(\mbz/2\mbz)$, and define the maps $\psi_2$ and $\psi_3$ as follows:
\[
\begin{tikzcd}
\psi_2 : \GL_2(\mbz/2\mbz) \rar{\can} & \frac{\GL_2(\mbz/2\mbz)}{\left\langle \gamma \right\rangle} \rar{\simeq} & \{ \pm 1 \}, \\
\psi_3 : \left\{ \begin{pmatrix} * & * \\ 0 & * \end{pmatrix} \right\} \rar{\det} & (\mbz/3\mbz)^\times \rar{\simeq} & \{ \pm 1 \}.
\end{tikzcd}
\]
We define the index eight subgroup $G(6) \subseteq \GL_2(\mbz/6\mbz)$ to be the fibered product
\begin{equation} \label{defofGof6}
G(6) := \GL_2(\mbz/2\mbz) \times_\psi \left\{ \begin{pmatrix} * & * \\ 0 & * \end{pmatrix} \right\}
\end{equation}
and define $G := \pi_{\GL_2}^{-1}(G(6)) \subseteq \GL_2(\hat{\mbz})$ to be the associated open subgroup.  
Let $E$ be an elliptic curve satisfying $\rho_E(G_\mbq) \, \dot\subseteq \, G$.
Our next theorem determines precisely the conditions under which $E$ is a $G$-Serre curve\footnote{We have $\left[ G \cap \SL_2(\hat{\mbz}) : [G,G] \right] = 2$, and thus, $E$ is a $G$-Serre curve if and only if $\rho_E(G_\mbq)$ is an index two subgroup of $G$.}.  First, let us denote by 
\[
\begin{split}
    \borel(\ell) :=& \left\{ \begin{pmatrix} * & * \\ 0 & * \end{pmatrix} \right\} \subseteq \GL_2(\mbz/\ell\mbz), \quad\quad\quad\quad\quad\quad\quad\quad\quad\;\; \left( \ell \text{ prime} \right), \\
    \mc{N}_{\ns}(\ell) :=& \left\{ \begin{pmatrix} x & -y \\ y & x \end{pmatrix} \right\} \cup \left\{ \begin{pmatrix} x & y \\ y & -x \end{pmatrix} \right\} \subseteq \GL_2(\mbz/\ell\mbz) \quad\quad \left( \ell \geq 3 \text{ prime} \right)
\end{split}
\]
respectively, the Borel subgroup and the Normalizer of a non-split Cartan subgroup of $\GL_2(\mbz/\ell\mbz)$.  Next, we define the following subgroups of $\GL_2(\hat{\mbz})$.
\begin{equation} \label{GSerrecurvedetectionsubgroups}
\begin{split}
G_{2,1} :=& \left\{ g \in \GL_2(\hat{\mbz}) : g \mod 2 \in \borel(2) \right\}, \\
G_{3,1} :=& \left\{ g \in \GL_2(\hat{\mbz}) : g \mod 3 \in \mc{N}_{\ns}(3) \right\}, \\
G_{4,1} :=& \left\{ g \in \GL_2(\hat{\mbz}) : g \mod 4 \in \left\langle \begin{pmatrix} 1 & 1 \\ 1 & 2 \end{pmatrix}, \begin{pmatrix} 0 & 1 \\ 3 & 0 \end{pmatrix}, \begin{pmatrix} 1 & 1 \\ 0 & 3 \end{pmatrix} \right\rangle \right\}, \\
G_{6,1} :=& \left\{ g \in \GL_2(\hat{\mbz}) : g \mod 6 \in \left\langle \begin{pmatrix} 1 & 1 \\ 0 & 5 \end{pmatrix}, \begin{pmatrix} 5 & 1 \\ 3 & 2 \end{pmatrix}, \begin{pmatrix} 5 & 4 \\ 4 & 1 \end{pmatrix} \right\rangle \right\}, \\
G_{9,1} :=& \left\{ g \in \GL_2(\hat{\mbz}) : g \mod 9 \in \left\langle \begin{pmatrix} 4 & 2 \\ 3 & 4 \end{pmatrix}, \begin{pmatrix} 2 & 0 \\ 0 & 5 \end{pmatrix}, \begin{pmatrix} 1 & 0 \\ 0 & 2 \end{pmatrix} \right\rangle \right\}, \\
G_{9,2} :=& \left\{ g \in \GL_2(\hat{\mbz}) : g \mod 9 \in \left\langle \begin{pmatrix} 1 & 1 \\ 0 & 1 \end{pmatrix}, \begin{pmatrix} 2 & 0 \\ 0 & 5 \end{pmatrix}, \begin{pmatrix} 1 & 0 \\ 0 & 2 \end{pmatrix} \right\rangle \right\}, \\
G_{9,3} :=& \left\{ g \in \GL_2(\hat{\mbz}) : g \mod 9 \in \left\langle \begin{pmatrix} 2 & 2 \\ 0 & 4 \end{pmatrix}, \begin{pmatrix} 4 & 7 \\ 0 & 8 \end{pmatrix}, \begin{pmatrix} 5 & 4 \\ 3 & 4 \end{pmatrix} \right\rangle \right\}, \\
G_{18,1} :=& \left\{ g \in \GL_2(\hat{\mbz}) : g \mod 18 \in \left\langle \begin{pmatrix} 7 & 17 \\ 0 & 5 \end{pmatrix}, \begin{pmatrix} 17 & 3 \\ 3 & 14 \end{pmatrix}, \begin{pmatrix} 4 & 3 \\ 3 & 14 \end{pmatrix} \right\rangle \right\}, \\
G_{18,2} :=& \left\{ g \in \GL_2(\hat{\mbz}) : g \mod 18 \in \left\langle \begin{pmatrix} 1 & 10 \\ 3 & 11 \end{pmatrix}, \begin{pmatrix} 16 & 3 \\ 9 & 8 \end{pmatrix}, \begin{pmatrix} 11 & 4 \\ 12 & 11 \end{pmatrix} \right\rangle \right\}, \\
G_{18,3} :=& \left\{ g \in \GL_2(\hat{\mbz}) : g \mod 18 \in \left\langle \begin{pmatrix} 16 & 9 \\ 9 & 8 \end{pmatrix}, \begin{pmatrix} 5 & 16 \\ 6 & 5 \end{pmatrix}, \begin{pmatrix} 7 & 13 \\ 3 & 10 \end{pmatrix} \right\rangle \right\}.
\end{split}
\end{equation}
By \eqref{ifEisnotaGSerrecurveequivalence}, we see that, when $E$ is not a $G$-Serre curve, $\rho_E(G_\mbq) \, \dot\subseteq \, H$ for some $H \in \mf{S}^{\max}(G)$.  Assuming that such a group $H$ satisfies $H(\ell) = \GL_2(\mbz/\ell\mbz)$ for each $\ell \notin \{ 2, 3 \}$, \cite[Theorem 2.7 \& Remark 2.8]{jones2} establishes in this case that $m_{GL_2}(H)$ must divide $216$.  Thus, it becomes a finite search to determine the set $\mf{S}^{\max}(G)$, and, carrying out this computation, we arrive at the following theorem.
\begin{Theorem} \label{applicationthm}
Let $G(6) \subseteq \GL_2(\mbz/6\mbz)$ be the index two subgroup defined by \eqref{defofGof6} and let $G = \pi_{\GL_2}^{-1}(G(6))$ be the associated open subgroup of $\GL_2(\hat{\mbz})$.  For each elliptic curve $E$ over $\mbq$ for which $\rho_E(G_\mbq) \, \dot\subseteq \, G$, we have that $E$ is \emph{not} a $G$-Serre curve if and only if
\begin{enumerate}
    \item there exists a group $G_{i,j}$ appearing in \eqref{GSerrecurvedetectionsubgroups} for which  $\rho_E(G_\mbq) \, \dot\subseteq \, G_{i,j}$, or
    \item there exists a prime $\ell \geq 5$ for which $\rho_{E,\ell}(G_\mbq) \neq \GL_2(\mbz/\ell\mbz)$.
\end{enumerate}
\end{Theorem}
\begin{remark}
The subgroups $G_{6,1}$ and $G_{18,1}$ of \eqref{GSerrecurvedetectionsubgroups} are the subgroups $G_6$ and $G_{18}$, respectively, that appear in Theorem \ref{mainthm}.  In particular, Theorem \ref{applicationthm} highlights the role played by non-abelian entanglement groups in this problem.  The group $G_{4,1}$ has appeared in various previous papers (see \cite{jonesaa}, \cite{dokchitser} and \cite{sutherlandzywina}); the groups $G_{9,1}$ and $G_{9,2}$ correspond, respectively, to the curves labeled $9C^ 0-9a$ and $9B^0-9a$ in the Table $1$ of \cite{sutherlandzywina}.
\end{remark}
\begin{remark} \label{maximalsubgroupsremark}
In the language introduced in \eqref{defsofSofGandSsupmaxofG}, we have that
\[
\mf{S}^{\max}(G) = \{ H \cap G : H = G_{i,j} \text{ as in \eqref{GSerrecurvedetectionsubgroups}} \} \cup \bigcup_{\ell \geq 5} \{ H \cap G :  H \in \mf{S}^{\max}(\ell) \}, 
\]
where the set $\mf{S}^{\max}(\ell)$ is defined as follows:  we let $\mf{S}(\ell)$ denote the set of all admissible open subgroups $H \subseteq \GL_2(\hat{\mbz})$ for which $H(\ell) \neq \GL_2(\mbz/\ell\mbz)$ and define $\mf{S}^{\max}(\ell) \subseteq \mf{S}(\ell)$ to be the subset of those $H \in \mf{S}(\ell)$ that are maximal with respect to subset inclusion. The genera of the modular curves $X_{G_{i,j}}$ and $X_{G \cap G_{i,j}}$ associated to each of the groups $G_{i,j}$ in \eqref{GSerrecurvedetectionsubgroups} are listed in Table \ref{tableofgenera}.
\begin{table}
\[
\begin{array}{|c||c|c|c|c|c|c|c|c|c|c|} \hline G_{i,j} \text{ from } \eqref{GSerrecurvedetectionsubgroups} & G_{2,1} & G_{3,1} & G_{4,1} & G_{6,1} & G_{9,1} & G_{9,2} & G_{9,3} & G_{18,1} & G_{18,2} & G_{18,3} \\ 
\hline \hline \genus(X_{\tilde{G}_{i,j}}) & 0 & 0 & 0 & 0 & 0 & 0 & 1 & 0 & 1 & 2 \\ 
\hline \genus(X_{\tilde{G} \cap \tilde{G}_{i,j}}) & 0 & 1 & 1 & 0 & 0 & 1 & 2 & 0 & 1 & 2 \\ 
\hline
\end{array}
\]
\vspace{.1in}
\caption{Genera of modular curves associated to $G_{i,j}$ from \eqref{GSerrecurvedetectionsubgroups} }
\label{tableofgenera}
\end{table}

\end{remark}
We now turn to the question:  In case $\rho_E(G_\mbq) \, \dot\subseteq \, G$, how \emph{likely} is it that $E$ is a $G$-Serre curve?  More precisely, assume for simplicity that $-I \in G$, and suppose that the modular curve $X_G$ has genus zero and that $X_G(\mbq) \neq \emptyset$.  Projecting from any rational point then yields a \emph{generic} point in $X_G\left( \mbq(t) \right)$, whose specializations give rise to all points in $X_G(\mbq)$.  Applying the forgetful map to the $j$-line
\[
j_G : \mbp^1_\mbq(t) \simeq X_G \longrightarrow X(1) \simeq \mbp^1_\mbq(j),
\]
we may then construct a Weierstrass model $\mc{E}$ defined over $\mbq(t)$:
\[
\mc{E} : y^2 = x^3 + a(t) x + b(t) \quad\quad \left( a(t), b(t) \in \mbq(t) \right),
\]
with $j$-invariant $j_{\mc{E}}(t) = j_G(t)$.  The generic Galois representation
\[
\rho_{\mc{E}} : G_{\mbq(t)} \longrightarrow \GL_2(\hat{\mbz})
\]
satisfies $\rho_{\mc{E}}(G_{\mbq(t)}) \, \dot\subseteq\, G$.  One can show independently that in fact, $\rho_{\mc{E}}(G_{\mbq(t)}) \doteq G$, but this may also be deduced from the following argument.

We are interested in understanding the nature of the specializations of $\mc{E}$ and their associated Galois representations.  More precisely, let $\gD_{\mc{E}}(t) \in \mbq(t)$ denote the discriminant of $\mc{E}$ and define the finite subset $B_{\mc{E}} \subseteq \mbq$ by
\begin{equation} \label{defofBE}
B_{\mc{E}} := \left\{ t_0 \in \mbq : a(t) \text{ or } b(t) \text{ is not regular at $t_0$}, \text{ or } \gD_{\mc{E}}(t_0) = 0, \text{ or } j_{\mc{E}}(t_0) \in \{ 0, 1728 \} \right\}.
\end{equation}
For each $t_0 \in \mbq - B_{\mc{E}}$, we denote by $\mc{E}_{t_0}$ the specialized Weierstrass model
$
y^2 = x^3 + a(t_0) x + b(t_0),
$
which is an elliptic curve over $\mbq$.  
We always have $\rho_{\mc{E}_{t_0}}(G_\mbq) \, \dot\subseteq \, G$, and we would like to ask how \emph{often} a specialization $\mc{E}_{t_0}$ is a $G$-Serre curve.  More precisely,
for $t_0 \in \mbq$, we denote by
\[
H(t_0) := \max \left\{ \left| x_0 \right|, \left| y_0 \right| : t_0 = \frac{x_0}{y_0} \text{ in lowest terms} \right\}.
\]
A standard exercise in analytic number theory (see for instance \cite[Theorem 3.9]{apostol}) shows that
\begin{equation} \label{asymptoticforrationalsofboundedheight}
\left| \{ t_0 \in \mbq : H(t_0) \leq T \} \right| \sim \frac{1}{2\zeta(2)} T^2 \quad \text{ as } \quad T \longrightarrow \infty.
\end{equation}
It is reasonable to ask whether the ratio
\[
\frac{\left| \left\{ t_0 \in \mbq - B_{\mc{E}} : H(t_0) \leq T, \, \mc{E}_{t_0} \text{ is a $G$-Serre curve} \right\} \right|}{\left| \left\{ t_0 \in \mbq - B_{\mc{E}} : H(t_0) \leq T \right\} \right|}
\]
tends to $1$ as $T \rightarrow \infty$.  As discussed in \cite{jones2} (see also \cite{cogrjo}), this is indeed the case; in the spirit of \cite{grant} and \cite{cogrjo}, one might want to compute an asymptotic formula as $T \rightarrow \infty$ for the size of the truncated exceptional set 
\begin{equation} \label{theexceptionalset}
\mc{S}(T) := \left\{ t_0 \in \mbq - B_{\mc{E}} : H(t_0) \leq T, \, \mc{E}_{t_0} \text{ is not a $G$-Serre curve} \right\}.
\end{equation}
By \eqref{ifEisnotaGSerrecurveequivalence}, we see that
\begin{equation} \label{infiniteunion}
\mc{S}(T) = \bigcup_{H \in \mf{S}^{\max}(G)} \mc{S}_H(T),
\end{equation}
where 
\[
\mc{S}_H(T) := \left\{ t_0 \in \mbq - B_{\mc{E}} :  H(t_0) \leq T, \; \rho_{\mc{E}_{t_0}}(G_\mbq) \, \dot\subseteq \, H  \right\}.
\]
A straightforward commutator calculation shows that, for any subgroup $H \subseteq G$ with $\tilde{H} = G$, $[H,H] = [G,G]$.  Thus, for any $H \in \mf{S}^{\max}(G)$, we must have $\tilde{H} \neq G$.  This observation shows that the sets $\mc{S}_H(T)$ above are (truncations of) ``thin sets'' in the sense of \cite{serre3}.  Indeed, for $H \in \mf{S}^{\max}(G)$, let
$
f_H : X_H \longrightarrow X_G
$
denote the forgetful map and 
$
d_H := \deg f_H = [G:H]
$
its degree.  We have the commuting diagram
\[
\begin{tikzcd}
X_H \arrow[black, bend left]{rr}{j_H} \rar{f_H} & X_G \rar{j_G} & \mbp^1_\mbq(j).
\end{tikzcd}
\]
For each $t_0 \in \mbq - B_{\mc{E}}$, we have
\[
\rho_{\mc{E}_{t_0}}(G_\mbq) \, \dot\subseteq \, H \; \Longleftrightarrow \; j_{\mc{E}_{t_0}} \in j_H(X_H(\mbq)),
\]
and thus
\[
\mc{S}_H(T) := \left\{ t_0 \in \mbq - B_{\mc{E}} :  H(t_0) \leq T, \; j_{\mc{E}_{t_0}} \in j_H\left(X_H(\mbq)\right)  \right\}.
\]
It follows from this (see \cite[p. 133]{serre3}) that, as $T \rightarrow \infty$, we have
\begin{equation} \label{generalasymptoticgrowth}
\left| \mc{S}_H(T) \right| \quad 
\begin{cases}
\sim C_{H} T^{2/d_{H}} & \text{ if } \genus(X_H) = 0 \text{ and } |X_H(\mbq)| = \infty, \\
\sim C_{H} (\log T)^{\rho_H/2} & \text{ if } \genus(X_H) = 1 \text{ and } |X_H(\mbq)| = \infty, \\
\ll_H 1 & \text{ if } \genus(X_H) \geq 2 \text{ or if } \genus(X_H) \leq 1 \text{ and } |X_H(\mbq)| < \infty,
\end{cases}
\end{equation}
where $C_{H} > 0$ denotes a constant, $d_{H}$ is the degree of $f_H$ in case $\genus(X_H) = 0$, and $\rho_H \geq 1$ denotes the Mordell-Weil rank of $X_H$ in case $\genus(X_H) = 1$ and $\left| X_H(\mbq) \right| = \infty$.  Furthermore, in \cite{cogrjo} it is shown that the infinite tail occurring in \eqref{infiniteunion} may be bounded, so that for any $\ve > 0$, we may write
\[
\mc{S}(T) = \mc{S}'(T) \cup \bigcup_{{\begin{substack} { H \in \mf{S}^{\max}(G) \\ m_{\GL_2}(H) \leq r } \end{substack}}} \mc{S}_H(T)
\]
for some $r = r_{\mc{E},\ve} \in \mbn$, where
\[
\mc{S}'(T) = 
\begin{cases}
    O_{\mc{E},\ve}\left( T^{1+\ve} \right) & \text{ if } \exists \begin{pmatrix} a & b \\ c & d \end{pmatrix} \in \GL_2(\mbq) \text{ and } P(x) \in \mbz[x] \text{ with } j_{\mc{E}}(t) = P\left(\frac{at+b}{ct+d}\right), \\
    O_{\mc{E},\ve}\left( T^{\ve} \right) & \text{ otherwise.}
\end{cases}
\]

Thus, we see that the asymptotic growth in $T$ of the truncated exceptional set $\mc{S}(T)$ is governed by the arithmetic of the curves $X_H$ for subgroups $H \in \mf{S}^{\max}(G)$; in particular, such growth is governed by those $H \in \mf{S}^{\max}(G)$ satisfying $\genus(X_H) = 0$, if such subgroups exist; we refine \eqref{defsofSofGandSsupmaxofG} by defining
\begin{equation*}
\mf{S}(G,g) := \left\{ H \in \mf{S}(G) : \, \genus(X_H) = g \right\}.
\end{equation*}
In case $\mf{S}(G,0) \neq \emptyset$, \eqref{generalasymptoticgrowth} leads us to the definitions
\[
\begin{split}
d_{\min}(G,0) :=& \min \left\{ d_H : \, H \in \mf{S}(G,0) \right\}, \\
\mf{S}_{\min}(G,0) :=& \left\{ H \in \mf{S}(G,0) \text{ and } d_H = d_{\min}(G,0) \right\}.
\end{split}
\]
In our case of the group $G$ appearing in Theorem \ref{applicationthm}, a computation shows that, for each prime $\ell \geq 5$ and each $H \in \mf{S}^{\max}(\ell)$ (see Remark \ref{maximalsubgroupsremark}), the modular curve $X_{\tilde{G} \cap \tilde{H}}$ has genus at least one.  A bit more computation shows that $d_{\min}(G,0) = 3$ and
\[
\mf{S}_{\min}(G,0) = \{ G_{2,1}, G_{6,1}, G_{9,1}, G_{18,1} \}
\]
(see Table \ref{tableofgenera}).  Finally, by considering the divisor $\Div(j_{\mc{E}})$, it is straightforward to verify that $j_{\mc{E}}(t)$ is not of the form $\ds P\left(\frac{at+b}{ct+d}\right)$ for any $\begin{pmatrix} a & b \\ c & d \end{pmatrix} \in \GL_2(\mbq)$ and any $P(x) \in \mbz[x]$.  Thus, the above analysis leads to the
following second theorem.  We define the rational functions
\begin{equation} \label{defoftsuboneandtsubtwo}
t_2(u) := \frac{3u^2 + 1}{u(u^2+3)}, \quad
t_6(u) := u^3+1, \quad t_9(u) := 1/u^3, \quad t_{18}(u) := \frac{1}{u^3-1},
\end{equation}
and the elliptic curve
\begin{equation} \label{defofGmcE}
    \mc{E} : \; y^2 = x^3 - \frac{108(t^2-1)(t^2-9)^3}{(t^4 + 18t^2 - 27)^2}x - \frac{432(t^2-1)(t^2-9)^3}{(t^4 + 18t^2 - 27)^2}.
\end{equation}
\begin{Theorem} \label{application2thm}
Let $G(6) \subseteq \GL_2(\mbz/6\mbz)$ be the index eight subgroup defined by \eqref{defofGof6} and let $G = \pi_{\GL_2}^{-1}(G(6))$ be the associated open subgroup of $\GL_2(\hat{\mbz})$. Let $\mc{E}$ be the elliptic curve over $\mbq(t)$ defined by \eqref{defofGmcE} and define the corresponding finite subset $B_{\mc{E}} \subset \mbq$ by \eqref{defofBE}.  We have
\begin{enumerate}
    \item For any $t_0 \in \mbq - B_{\mc{E}}$, the specialized curve $\mc{E}_{t_0}$ satisfies
    \[
    \rho_{\mc{E}_{t_0}}(G_\mbq) \, \dot\subseteq \, G.
    \]
    \item For any $\ve > 0$, the truncated exceptional set $\mc{S}(T)$, defined in \eqref{theexceptionalset}, satisfies
    \[
    | \mc{S}(T) | = C T^{2/3} + O_{\ve}(T^{\ve}),
    \]
    for some constant $C > 0$. More precisely, we have
    \begin{equation} \label{thewholepoint}
    \mc{S}(T) = \mc{S}'(T) \cup \bigcup_{ i \in \{ 2, 6, 9, 18 \}} \{ t_0 \in \mbq - B_{\mc{E}} : H(t_0) \leq T, t_0 = t_i(u_0) \text{ for some } u_0 \in \mbq \},
    \end{equation}
    where $t_2(u), t_6(u), t_9(u), t_{18}(u) \in \mbq(u)$ are as in \eqref{defoftsuboneandtsubtwo} and the set $\mc{S}'(T)$ satisfies $|\mc{S}'(T)| = O_{\ve}(T^{\ve})$.
    \item In particular, we have
    \[
    \lim_{T \rightarrow \infty} \frac{\left| \left\{ t_0 \in \mbq - B_{\mc{E}} : H(t_0) \leq T \text{ and } \mc{E}_{t_0} \text{ is a $G$-Serre curve} \right\} \right| }{\left| \left\{ t_0 \in \mbq - B_{\mc{E}} : H(t_0) \leq T \right\} \right|} = 1,
    \]
    i.e. almost all specializations of $\mc{E}$ are $G$-Serre curves.
\end{enumerate}
\end{Theorem}

\begin{remark}
Part (3) of Theorem \ref{applicationthm} follows immediately from part (2) and \eqref{asymptoticforrationalsofboundedheight}. It is also is a special case of \cite[Theorem 2.11]{jones2},  the relevance to this paper being \eqref{thewholepoint}, which highlights the involvement of non-abelian entanglement modular curves in this problem, since $t_6(u)$ (resp. $t_{18}(u)$) corresponds to the forgetful map associated to the curve $X_{G_6 \cap G}$ (resp. the curve $X_{G_{18}}$) featured in Theorem \ref{mainthm}:
\[
\begin{tikzcd}
X_{G_6 \cap G} \arrow[black, bend left]{rrr}{\text{forgetful map}} \rar{\simeq} & \mbp^1_\mbq(u) \rar{t_6} & \mbp^1_\mbq(t) \rar{\simeq} & X_G, \\
X_{G_{18}} \arrow[black, bend right]{rrr}{\text{forgetful map}} \rar{\simeq} & \mbp^1_\mbq(u) \rar{t_{18}} & \mbp^1_\mbq(t) \rar{\simeq} & X_G.
\end{tikzcd}
\]
In particular, we see that non-abelian entanglement groups arise naturally in the problem of determining and counting elliptic curves over $\mbq$ for which $\rho_E(G_\mbq) \, \dot\subseteq \, G$ is as large as possible, given the constraints dictated by the Kronecker-Weber theorem.
\end{remark}

\bigskip

\section{Reducing to a Finite Search}\label{Section:FiniteSearch}

In this section we prove Proposition \ref{keypropositionformainthm}, which reduces the computation of $\mc{G}_{\nonab}^{\max}(0) / \doteq$ to a finite search.  We consider the quotient set 
\[
\mc{G}_{\nonab}(0) / \doteq
\]
of all open subgroups $G \subseteq \GL_2(\hat{\mbz})$ of genus zero that are non-abelian entanglement groups in the sense of Definition \ref{defofnonabelianentanglementgroup}, up to conjugation inside $\GL_2(\hat{\mbz})$.  More generally, we may fix an arbitrary genus $g$ and consider the quotient set $\mc{G}(g) / \doteq$.  In what follows, we will be discussing the corresponding set of modular curves
$
\{ X_{\tilde{G}} : G \in \mc{G}(g) / \doteq \}
$,
so it will be natural to introduce the notation
\[
\tilde{\mc{G}} := \{ \tilde{G} : G \in \mc{G} \}, \quad\quad \tilde{\mc{G}}(g) := \{ \tilde{G} : G \in \mc{G}(g) \},
\]
and the associated set of modular curves
\[
\{ X_{\tilde{G}} : \tilde{G} \in \tilde{\mc{G}} / \doteq \}.
\]
If we view these modular curves \emph{geometrically}, i.e. if we regard two such curves as equivalent if they are isomorphic over $\ol{K}$, then the further quotient set 
\begin{equation} \label{modularcurvesgeometrically}
\{ X_{\tilde{G}} : \tilde{G} \in \tilde{\mc{G}}(g) / \doteq \} / \simeq_{\ol{K}}
\end{equation}
 of geometric modular curves is finite, for any fixed $g \in \mbz_{\geq 0}$.  This may be restated in terms of the groups $\tilde{G} \in \tilde{\mc{G}}(g) / \doteq$ as follows.  We have
\[
X_{\tilde{G}_1} \simeq_{\ol{K}} X_{\tilde{G}_2} \; \Longleftrightarrow \; \tilde{G}_1 \cap \SL_2(\hat{\mbz}) \doteq \tilde{G}_2 \cap \SL_2(\hat{\mbz}).
\]
Thus, coarsening the relation $\doteq$ to $\doteq_{\SL_2}$, defined by
\[
\tilde{G}_1 \doteq_{\SL_2} \tilde{G}_2 \; \myeq \; \tilde{G}_1 \cap \SL_2(\hat{\mbz}) \doteq \tilde{G}_2 \cap \SL_2(\hat{\mbz}),
\]
we have that
\begin{equation} \label{equivalenttorademacher}
   \left| \tilde{\mc{G}}(g) / \doteq_{\SL_2} \right| < \infty \quad\quad \left( g \in \mbz_{\geq 0} \right).
\end{equation}
However, within each $\doteq_{\SL_2}$-equivalence class, there are infinitely many $\doteq$-equivalence classes, which corresponds in part to the fact that any given $\ol{K}$-isomorphism class in
\eqref{modularcurvesgeometrically}
contains infinitely many twists, i.e. infinitely many $K$-isomorphism classes.  
The case $g=0$ of \eqref{equivalenttorademacher} is equivalent to a well-known conjecture of Rademacher that was proven by Denin (see \cite{denin1}, \cite{denin2} and \cite{denin3}).  More generally, in \cite{thompson} and \cite{zograf}, the same is shown for a general $g \in \mbz_{\geq 0}$. 
In addition, there is a fair amount of literature on the \emph{effective} resolution of Rademacher's conjecture.  In particular, Cummins and Pauli \cite{cumminspauli} have produced the complete list of the elements of $\tilde{\mc{G}}(g) / \doteq_{\SL_2}$
for $g \leq 24$; it can be seen in the tables therein that 
\begin{equation} \label{cumminspaulilistoflevels}
G \in \mc{G}(0) \; \Longrightarrow \; m_{\SL_2}(\tilde{G}) \in \left\{ \begin{matrix} 1, 2, 3, 4, 5, 6, 7, 8, 9, 10, 11, 12, 13, 14, 15, \\ 16, 18, 20, 21, 24, 25, 26, 27, 28, 30, 32, 36, 48 \end{matrix} \right\}.
\end{equation}
It is possible that $m_{\SL_2}(G) > m_{\SL_2}(\tilde{G})$, and the following lemma controls this difference.
\begin{lemma} \label{Mto2timesMlemma}
Let $G \subseteq \GL_2(\hat{\mbz})$ be an open subgroup.  We then have
\begin{equation} \label{Gversion}
\frac{m_{\SL_2}(G)}{m_{\SL_2}(\tilde{G})} \in \{ 1, 2 \},
\end{equation}
where $\tilde{G}$ is as in \eqref{defofGtilde}.
\end{lemma}
Lemma \ref{Mto2timesMlemma} will be established as a corollary of the Lemmas \ref{horizontallemma} and \ref{verticallemma} below, which are in turn aided by the next two lemmas.
\begin{lemma} \label{zywinalemma}
There is no proper subgroup $S \subsetneq \SL_2(\hat{\mbz})$ satisfying $\tilde{S} = \SL_2(\hat{\mbz})$.
\end{lemma}
\begin{proof}
See \cite[Lemma 2.3]{zywina2}
\end{proof}
\begin{lemma} \label{goursatlemma} (Goursat's Lemma)
\noindent
Let $G_1$, $G_2$ be groups and for $i \in \{1, 2 \}$ denote by $\pr_i : G_1 \times G_2 \longrightarrow G_i$ the projection map onto the $i$-th factor. Let $G \subseteq G_1 \times G_2$ be a subgroup and 
assume that 
$$
\pr_1(G) = G_1, \; \pr_2(G) = G_2.
$$
Then 
there exists a group $\Gamma$ together with a pair of surjective homomorphisms 
\[
\begin{split}
\psi_1 : G_1 &\longrightarrow \Gamma \\
\psi_2 : G_2 &\longrightarrow \Gamma
\end{split}
\]
so that 
\[
G = G_1 \times_\psi G_2 := \{ (g_1,g_2) \in G_1 \times G_2 : \psi_1(g_1) = \psi_2(g_2) \}.
\]
\end{lemma}
\begin{proof}
See \cite[Lemma (5.2.1)]{ribet}.
\end{proof}
\begin{lemma} \label{horizontallemma}
Suppose that $G \subseteq \GL_2(\hat{\mbz})$ is an open subgroup satisfying $m_{\SL_2}(\tilde{G}) < m_{\SL_2}(G)$, and that there is a prime $p$ dividing $m_{\SL_2}(G)$ that doesn't divide $m_{\SL_2}(\tilde{G})$.  Then $p = 2$ (so $m_{\SL_2}(\tilde{G})$ is odd) and $m_{\SL_2}(G) = 2m_{\SL_2}(\tilde{G})$.
\end{lemma}
\begin{proof}
First, we set $S := G \cap \SL_2(\hat{\mbz})$, so that $\tilde{S} = \tilde{G} \cap \SL_2(\hat{\mbz})$ as well; furthermore, we make the abbreviations
\begin{equation} \label{defofmsubSandmsubStilde}
\begin{split}
m_S :=& m_{\SL_2}(S) \quad\quad\quad m_{\tilde{S}} := m_{\SL_2}(\tilde{S}) \\
=& m_{\SL_2}(G) \quad\quad\quad\quad\;\; \,= m_{\SL_2}(\tilde{G}).
\end{split}
\end{equation}
Let $p$ be a prime as in the statement of the lemma, let $p^\ga$ be the exact power of $p$ dividing $m_S$ and let us write $m_S = p^\ga m_S'$, where $p \nmid m_S'$ and $m_{\tilde{S}} \mid m_S'$.  By definition of $m_{\tilde{S}}$, under the isomorphism of the Chinese Remainder Theorem, we have
\[
S(m_S) \subseteq \tilde{S}(m_S) \simeq \tilde{S}(m_S') \times \SL_2(\mbz/p^\ga\mbz),
\]
In light of Lemma \ref{goursatlemma}, there are 3 possibilities for the index two subgroup $S(m) \subseteq \tilde{S}(m_S)$:
\begin{equation} \label{threepossibilities}
\begin{split}
S(m_S) &= S(m_S') \times \SL_2(\mbz/p^\ga \mbz) \quad\quad\quad\;\;\, [\tilde{S}(m_{S}') : S(m_{S}') ] = 2 \\
S(m_S) &= \tilde{S}(m_S') \times S(p^\ga) \quad\quad\quad\quad\;\, [\SL_2(\mbz/p^\ga\mbz) : S(p^\ga) ] = 2 \\
S(m_S) &= \tilde{S}(m_S') \times_\psi \SL_2(\mbz/p^\ga\mbz) \quad\quad | \psi_{p^\ga}(\SL_2(\mbz/p^\ga\mbz)) | = 2.
\end{split}
\end{equation}
The first possibility in \eqref{threepossibilities} would imply that $m_S$ divides $m_S'$, a contradiction.  The second possibility would imply the existence of a proper subgroup $S := \pi_{\SL_2}^{-1}\left( S(p^\ga) \right) \subsetneq \SL_2(\hat{\mbz})$ satisfying $\tilde{S} = \SL_2(\hat{\mbz})$, contradicting Lemma \ref{zywinalemma}.  We thus conclude that only the third possibility can occur:
\[
S(m_S) = \tilde{S}(m_S') \times_\psi \SL_2(\mbz/p^\ga\mbz) \quad\quad | \psi_{p^\ga}(\SL_2(\mbz/p^\ga\mbz)) | = 2.
\]
We now consider the map $\psi_{p^\ga} : \SL_2(\mbz/p^\ga\mbz) \longrightarrow \{ \pm 1 \}$.  Using the well-known fact that the abelianization map $\SL_2(\hat{\mbz}) \longrightarrow \mbz/12\mbz$ factors as 
\begin{equation} \label{abelianaztionmap}
\begin{tikzcd}
\SL_2(\hat{\mbz}) \rar{\red \times \red} & \SL_2(\mbz/4\mbz) \times \SL_2(\mbz/3\mbz) \rar{\ab \times \ab} & \mbz/4\mbz \times \mbz/3\mbz,
\end{tikzcd}
\end{equation}
it follows that $p = 2$ and that $\ker \psi_{p^\ga} = \pi_{\SL_2}^{-1}(A_3)$, where $A_3 \subseteq S_3 \simeq \SL_2(\mbz/2\mbz)$ is the unique subgroup of index 2.  Since the $\SL_2$-level of $S$ is $2^\ga m_S'$ and the map $\psi_{2^\ga}$ factors through reduction modulo $2$, we must then have $\ga = 1$.  To finish the proof, we will show that $m_S' = m_{\tilde{S}}$.  Suppose for the sake of contradiction that \begin{equation} \label{msubSprimegtmsubtildeS}
m_S' > m_{\tilde{S}}.  
\end{equation}
The hypothesis that $m_S = 2 m_S'$ implies that $\ker \psi_{m_S'} \subseteq \tilde{S}(m_S')$ is an index 2 subgroup whose image at any lower level $m_S''$ is all of $\tilde{S}(m_S'')$.  In particular, fixing any (necessarily odd) prime $p$ dividing $m_S'$, we have that $\ker \psi_{m_S'} (m_S'/p) = \tilde{S}(m_S'/p)$, and since $\ker \left( \SL_2(\mbz/p^2\mbz) \rightarrow \SL_2(\mbz/p\mbz) \right)$ is a $p$-group and $\ker \psi_{m_S'}(m_S') \subseteq \tilde{S}(m_S')$ is an index two subgroup, it follows that $m_S'$ must be square-free.  Thus, by \eqref{msubSprimegtmsubtildeS} there must be a square-free number $m_S''$ satisfying $m_S' = m_{\tilde{S}}m_S''$ and with
\begin{equation} \label{kerpsiasfiberedproduct}
\ker \psi_{m_S'} \simeq \tilde{S}(m_{\tilde{S}}) \times_{\phi} \SL_2(\mbz/m_S''\mbz),
\end{equation}
where the image of $\phi$ is a group of order $2$.  Again by \eqref{abelianaztionmap}, any non-trivial image of $\SL_2(\mbz/m_S''\mbz)$ must have order divisible by $3$, and thus the fibered product \eqref{kerpsiasfiberedproduct} is a full cartesian product, so that $\ker \psi_{m_S'} = \tilde{S}(m_S')$, a contradiction.
This implies that $m_S' = m_{\tilde{S}}$, proving the lemma.
\end{proof}
We now prove a lemma that handles the ``vertical'' situation.
\begin{lemma} \label{verticallemma}
Let $G \subseteq \GL_2(\hat{\mbz})$ be an open subgroup, and suppose that $m_{\SL_2}(\tilde{G}) < m_{\SL_2}(G)$ and that
\[
m_{\SL_2}(G) \mid m_{\SL_2}(\tilde{G})^\infty,
\]
i.e. that any prime $p$ dividing $m_{\SL_2}(G)$ must also divide $m_{\SL_2}(\tilde{G})$.  We then have that $m_{\SL_2}(\tilde{G})$ is even and $m_{\SL_2}(G) = 2m_{\SL_2}(\tilde{G})$.
\end{lemma}
\begin{proof}
As before, we set $S := G \cap \SL_2(\hat{\mbz})$, so that $\tilde{S} = \tilde{G} \cap \SL_2(\hat{\mbz})$, and also define $m_S$ and $m_{\tilde{S}}$ by \eqref{defofmsubSandmsubStilde}.  Since $[\tilde{S} : S] = 2$ and by definitions of $m_S$ and $m_{\tilde{S}}$, for any divisor $d$ of $m_S$, we have
\begin{equation} \label{SofdhasindextwointildeSofd}
m_{\tilde{S}} \mid d \mid m_S \text{ and } d < m_S \; \Longrightarrow \; S(d) = \tilde{S}(d).
\end{equation}
Let $p$ be any prime for which
\begin{equation} \label{msubtildeSdividesmsubSoverp}
m_{\tilde{S}} \; \text{ divides } \; \frac{m_S}{p}.
\end{equation}
Since the kernel
\[
\ker \left( \SL_2(\mbz/m_S\mbz) \rightarrow \SL_2(\mbz/(m_S/p)\mbz) \right)
\]
is an abelian $p$-group, it follows from \eqref{SofdhasindextwointildeSofd} that any $p$ satisfying \eqref{msubtildeSdividesmsubSoverp} must be even, and thus $m_S = 2^\ga m_{\tilde{S}}$ for some $\ga \geq 1$.  We now show that $\ga = 1$.  Note that each matrix $X$ in the set
\[
\mc{K} := \left\{ \begin{pmatrix} 0 & 1 \\ 0 & 0 \end{pmatrix}, \begin{pmatrix} 0 & 0 \\ 1 & 0 \end{pmatrix}, \begin{pmatrix} 1 & 1 \\ -1 & -1 \end{pmatrix} \right\} \subseteq M_{2\times 2}(\mbz)
\]
satisfies $X^2 = 0$, and also (recall that $2 \mid m_{\tilde{S}}$) that
\begin{equation} \label{shapeofkernelover2}
\ker\left( \SL_2(\mbz/2^nm_{\tilde{S}} \mbz) \rightarrow \SL_2(\mbz/2^{n-1}m_{\tilde{S}} \mbz) \right) = \left\langle \{ I + 2^{n-1}m_{\tilde{S}} X \mod 2^nm_{\tilde{S}} : X \in \mc{K} \}\right\rangle \quad\quad (n \geq 1).
\end{equation}
If $\ga > 1$ then, by \eqref{SofdhasindextwointildeSofd}, $\ker\left( \SL_2(\mbz/2m_{\tilde{S}} \mbz) \rightarrow \SL_2(\mbz/m_{\tilde{S}} \mbz) \right) \subseteq S(2m_{\tilde{S}})$.  Fixing any $X \in \mc{K}$, we then have 
\[
I + m_{\tilde{S}}X \mod 2m_{\tilde{S}} \in S(2m_{\tilde{S}}).  
\]
Replacing $X$ by an appropriate lift in $M_{2\times 2}(\mbz_2)$ of $X \mod 2$ (which must still satisfy $X^2 \equiv 0 \mod 2$), we may assume that $I + m_{\tilde{S}}X \mod 4 m_{\tilde{S}} \in S(4 m_{\tilde{S}})$, and so then
\[
(I + m_{\tilde{S}}X)^2 = I + 2m_{\tilde{S}} X + m_{\tilde{S}}^2 X^2 \equiv I + 2m_{\tilde{S}} X \mod 4m_{\tilde{S}} \in S(4m_{\tilde{S}}), 
\]
and by \eqref{shapeofkernelover2}, we then see that $\ker\left( \SL_2(\mbz/4m_{\tilde{S}} \mbz) \rightarrow \SL_2(\mbz/2m_{\tilde{S}} \mbz) \right) \subseteq S(4m_{\tilde{S}})$.  Continuing inductively, we would then conclude that $S(2^\ga m_{\tilde{S}}) = \pi_{\SL_2}^{-1}(S(m_{\tilde{S}})) = \tilde{S}(2^\ga m_{\tilde{S}})$, a contradiction.  Thus, we must have $S(2m_{\tilde{S}}) \subsetneq \tilde{S}(2m_{\tilde{S}})$, and so $m_S = 2m_{\tilde{S}}$, as asserted.
\end{proof}
\noindent \emph{Proof of Lemma \ref{Mto2timesMlemma}.}  
Lemma \ref{Mto2timesMlemma} follows immediately from Lemmas \ref{horizontallemma} and \ref{verticallemma}. \hfill $\Box$ \\

Applying Lemma \ref{Mto2timesMlemma} to \eqref{cumminspaulilistoflevels}, we obtain the following corollary.
\begin{cor} \label{levelincreaseboundcor}
Let $G \in \mc{G}(0)$.  We then have
\[
m_{\SL_2}(G) \in \left\{ \begin{matrix} 1, 2, 3, 4, 5, 6, 7, 8, 9, 10, 11, 12, 13, 14, 15, 16, 18, 20, 21, 22, 24, \\ 25, 26, 27, 28, 30, 32, 36, 40, 42, 48, 50, 52, 54, 56, 60, 64, 72, 96 \end{matrix} \right\}.
\]
\end{cor}
The proof of Proposition \ref{keypropositionformainthm} involves a group theoretical analysis together with a MAGMA computation. We now develop the group theory part.

\subsection{Lemmas on fibered products} \label{fiberedproductlemmassection}

Let $G_1$ and $G_2$ be groups, let
\[
\begin{split}
\phi_i : G_i \longrightarrow \Gamma_\phi \\
\psi_i : G_i \longrightarrow \Gamma_\psi
\end{split}
\]
be surjective group homomorphisms and let
\[
\begin{split}
G_1 \times_\phi G_2 &:= \{ (g_1,g_2) \in G_1 \times G_2 : \; \phi_1(g_1) = \phi_2(g_2) \}, \\
G_1 \times_\psi G_2 &:= \{ (g_1,g_2) \in G_1 \times G_2 : \; \psi_1(g_1) = \psi_2(g_2) \}
\end{split} 
\]
be the associated fibered products.  We call $\Gamma_{\phi}$ the \emph{common quotient} associated to $G_1 \times_\phi G_2$, and likewise with $\Gamma_{\psi}$.  

\begin{lemma} \label{thecorollary}
In the above setting, we have
\begin{equation} \label{thecorollaryequation}
G_1 \times_\phi G_2 = G_1 \times_\psi G_2 \; \Longleftrightarrow \; \ker \psi_1 \times \ker \psi_2 \subseteq G_1 \times_\phi G_2 \subseteq G_1 \times_\psi G_2.
\end{equation}
\end{lemma}
\begin{proof}
Since ``$\Rightarrow$'' is trivial, we prove the ``$\Leftarrow$'' direction.  The condition $\ker \psi_1 \times \ker \psi_2 \subseteq G_1 \times_\phi G_2$ implies that $\ker \psi_i \subseteq \ker\phi_i$ for each $i \in \{ 1, 2 \}$.  On the other hand, the containment $G_1 \times_\phi G_2 \subseteq G_1 \times_\psi G_2$ implies that $\ker \phi_1 \times \ker \phi_2 \subseteq G_1 \times_\psi G_2$, which forces $\ker \phi_i \subseteq \ker \psi_i$ for each $i$.  Thus we have
\[
\ker \phi_i = \ker \psi_i \quad\quad \left( i \in \{ 1, 2 \} \right).
\]
It follows that there are isomorphisms $\eta_i : \Gamma_{\phi} \rightarrow \Gamma_{\psi}$ such that, for each $i \in \{ 1, 2 \}$, we have $\psi_i = \eta_i \phi_i$.  Now, if there exists $\gamma \in \Gamma_\phi$ with $\eta_1(\gamma) \neq \eta_2(\gamma)$, then, choosing $g_1 \in G_1$ and $g_2 \in G_2$ with $\phi_i(g_i) = \gamma$, we find that $(g_1,g_2) \in G_1 \times_\phi G_2$ but $(g_1,g_2) \notin G_1 \times_\psi G_2$, a contradiction.  Thus, $\eta_1 = \eta_2 =: \eta$, and it follows that
\[
G_1 \times_\phi G_2 = G_1 \times_{(\eta \phi_1, \eta \phi_2)} G_2 = G_1 \times_\psi G_2,
\]
as asserted, establishing the ``$\Leftarrow$'' direction and proving the lemma.
\end{proof}
The following lemma is key throughout.  Let $G_1$ and $G_2$ be groups, together with surjective group homomorphisms
\[
\psi_i : G_i \longrightarrow \Gamma
\]
onto a common group $\Gamma$, and let $G_1 \times_\psi G_2$ be the corresponding fibered product.  For $i \in \{1, 2\}$, let $\pi_i : G_i \twoheadrightarrow \ol{G}_i$ be a surjective group homomorphisms (which we will denote by $g_i \mapsto \ol{g}_i$) and consider the induced surjection 
\[
\pi : G_1 \times G_2 \longrightarrow \ol{G}_1 \times \ol{G}_2, \quad (g_1,g_2) \mapsto (\ol{g}_1, \ol{g}_2)
\]
(in other words, $\pi := \pi_1 \times \pi_2$).  The following lemma describes explicitly the image of $G$ under $\pi$.  Define the quotient group $\ol{\Gamma}$ by
\[
\ol{\Gamma} := \frac{\Gamma}{\psi_1(\ker \pi_1)\psi_2(\ker \pi_2)},
\]
let $\varpi : \Gamma \longrightarrow \ol{\Gamma}$ be the canonical surjection and let $\ol{\psi}_i := \varpi \circ \psi_i$.  Note that $\ol{\psi}_i$ induces a well defined surjective homomorphism $\ol{G}_i \longrightarrow \ol{\Gamma}$ (via $\ol{g}_i \mapsto \ol{\psi}_i(g_i)$), which we will continue to denote by $\ol{\psi}_i$.  Note the functional equation
\begin{equation} \label{functionaleqnwithpsiandpsibar}
\varpi \circ \psi_i = \ol{\psi}_i \circ \pi_i.
\end{equation}
We let $\ol{\psi}$ denote the pair $(\ol{\psi}_1,\ol{\psi}_2)$ and $\ol{G}_1 \times_{\ol{\psi}} \ol{G}_2$ the corresponding fibered product group. 
\begin{lemma} \label{onesideprojectionlemma}
Let the $G_1$ and $G_2$ be groups and consider the fibered product $G_1 \times_\psi G_2$ as described above.  Then, with the notation just outlined, we have
\begin{equation*} 
\pi(G_1 \times_\psi G_2) = \ol{G}_1 \times_{\ol{\psi}} \ol{G}_2.
\end{equation*}
\end{lemma}
\begin{proof}
The containment ``$\subseteq$'' is immediate, since $\psi_1(g_1) = \psi_2(g_2)$ implies that $\varpi(\psi_1(g_1)) = \varpi(\psi_2(g_2))$, and so $\ol{\psi}_1(\ol{g}_1) = \ol{\psi}_2(\ol{g}_2)$.  Furthermore, it follows from the surjectivity of each $\pi_i$ that $\pi(G_1 \times_\psi G_2) \subseteq \ol{G}_1 \times \ol{G}_2$ is a subgroup that projects onto $\ol{G}_1$ and onto $\ol{G}_2$ via the canonical projections.  Thus, by Lemma \ref{goursatlemma},
$
\pi(G_1 \times_\psi G_2) = \ol{G}_1 \times_{\eta} \ol{G}_2
$
for some fibering maps $(\eta_1,\eta_2)$.  Furthermore, we claim that 
\begin{equation} \label{kernelcontainment}
\ker \ol{\psi}_1 \times \ker \ol{\psi}_2 \subseteq \ol{G}_1 \times_{\eta} \ol{G}_2.
\end{equation}
Indeed, it is sufficient to show that $\ker \ol{\psi}_1 \times \{ 1 \} \subseteq \ol{G}_1 \times_{\eta} \ol{G}_2$ and that $\{ 1 \} \times \ker \ol{\psi}_2 \subseteq \ol{G}_1 \times_{\eta} \ol{G}_2$.  Let $\ol{x}_1 \in \ker \ol{\psi}_1$ and let $x_1 \in G_1$ be any lift under $\pi_1$ of $\ol{x}_1$.  By definition of $\ol{\psi}_1$, we may adjust $x_1 \in \pi_1^{-1}(\ol{x}_1)$ so that $\psi_1(x_1) \in \psi_2(\ker \pi_2)$, and thus there exists $k_2 \in \ker \pi_2$ for which $(x_1,k_2) \in G_1 \times_\psi G_2$.  Applying $\pi$, it follows that $(\ol{x}_1,1) \in \ol{G}_1 \times_{\eta} \ol{G}_2$, and so $\ker \ol{\psi}_1 \times \{ 1 \} \subseteq \ol{G}_1 \times_{\eta} \ol{G}_2$; the argument that $\{ 1 \} \times \ker \ol{\psi}_2 \subseteq \ol{G}_1 \times_{\eta} \ol{G}_2$ is similar.  The containment \eqref{kernelcontainment} follows.

Having established that
\[
\ker \ol{\psi}_1 \times \ker \ol{\psi}_2 \subseteq \ol{G}_1 \times_{\eta} \ol{G}_2 \subseteq \ol{G}_1 \times_{\ol{\psi}} \ol{G}_2,
\]
Lemma \ref{thecorollary} now finishes the proof.
\end{proof}

Our final lemma has to do with intersecting fibered products with full cartesian products, and will later be applied to the situation of intersecting with $\SL_2$.  Let $G_1 \times_\psi G_2$ be a fibered product and let $S_i \subseteq G_i$ be subgroups.  It is clear that
\[
\left( G_1 \times_\psi G_2 \right) \cap \left( S_1 \times S_2 \right) = S_1 \times_\psi S_2,
\]
but the canonical projection maps in the right-hand expression may not be surjective, which can cause confusion.  To remedy this, let us say that $\Gamma$ is the common quotient group associated to the fibered product $G_1 \times_\psi G_2$ and put
\[
\Gamma_S := \psi_1(S_1) \cap \psi_2(S_2).
\]
\begin{lemma} \label{SL2contractionlemma}
Let $G_1 \times_\psi G_2$ be a fibered product and let $S_i \subseteq G_i$ be subgroups.  Then
\[
\left( G_1 \times_\psi G_2 \right) \cap \left( S_1 \times S_2 \right) = \psi_1 \vert_{S_1}^{-1}(\Gamma_S)\times_\psi \psi_2 \vert_{S_2}^{-1}(\Gamma_S),
\]
and the canonical projection maps in the right-hand expression are surjective.  Moreover, 
\[
\psi_1\left( \psi_1 \vert_{S_1}^{-1}(\Gamma_S) \right) = \Gamma_S = \psi_2 \left( \psi_2 \vert_{S_2}^{-1}(\Gamma_S) \right).
\]

\end{lemma}

\subsection{Pre-twist groups and how we search for them}

We will now define the notion of a pre-twist group, as a means to aid in the search for $G \in \mc{G}$ which satisfy $m_{\SL_2}(G) < m_{\GL_2}(G)$. 
Our goal to prove Proposition \ref{keypropositionformainthm} may be stated more broadly as follows.
\begin{goal}
To find all (maximal) non-abelian entanglement groups of genus $0$ (or, more generally, of fixed genus $g \geq 0$).
\end{goal}
For any given non-abelian entanglement group $G$, either $m_{\GL_2}(G) = m_{\SL_2}(G)$ or not.  If $m_{\GL_2}(G) = m_{\SL_2}(G)$, then $G$ will be found when we search through all groups with $\GL_2$-level appearing in the list from Corollary \ref{levelincreaseboundcor}.  If $m_{\GL_2}(G) \neq m_{\SL_2}(G)$, then we will view $G$ as being a \emph{twist cover} of some group $\ol{G}$ whose $\GL_2$-level appears in that list. 

\begin{Definition} \label{pretwistgroupdefinition}
A subgroup $\ol{G} \subseteq \GL_2(\hat{\mbz})$ is called a \textbf{pre-twist group} if there exists a non-abelian entanglement group $G \subsetneq \ol{G}$ such that 
\begin{equation} \label{twistcoverconditions}
\begin{split}
m_{\SL_2}(G) &= m_{\GL_2}(\ol{G}) =: \ol{m}, \\
m_{\GL_2}(G) &=: m > \ol{m}, \quad \text{ and} \\
G(\ol{m}) &= \ol{G}(\ol{m}).
\end{split}
\end{equation}
If $\ol{G}$ is a pre-twist group, then a \textbf{twist cover of $\ol{G}$} refers to any non-abelian entanglement group $G \subsetneq \ol{G}$ satisfying \eqref{twistcoverconditions} .
\end{Definition}
If $G$ is a non-abelian entanglement group with $m_{\GL_2}(G) > m_{\SL_2}(G) =: \ol{m}$, then we define
\[
\ol{G} := \pi_{\GL_2}^{-1}(G(\ol{m})) \subseteq \GL_2(\hat{\mbz}).
\]
Clearly $\ol{G}$ is a pre-twist group and $G$ is a twist cover of $\ol{G}$.  Thus, to find all non-abelian entanglement groups whose $\GL_2$-level and $\SL_2$-level are different, it suffices to first find all pre-twist groups $\ol{G}$ and then describe the process of constructing twist covers $G$ of $\ol{G}$.  

Our next lemma will aid in the proof of Proposition \ref{pretwistscenarioproposition} below, which in turn implies a somewhat restrictive necessary condition on pre-twist groups that allows us to deduce Proposition \ref{keypropositionformainthm}.  First we observe two elementary lemmas about twist covers.  The set-up is as follows: $\ol{G}$ will be a pre-twist group of $\GL_2$-level $\ol{m}$ and $G \subsetneq \ol{G}$ will be a twist cover of $\ol{G}$ of $\GL_2$-level $m > \ol{m}$.  Suppose that $m = m_1 m_2$ with $\gcd(m_1,m_2) = 1$, that $\ol{m}_i := \gcd(\ol{m},m_i)$ and that
\[
\begin{split}
G(m) &= G(m_1) \times_{\psi} G(m_2), \\
G(\ol{m}) &= G(\ol{m}_1) \times_{\ol{\psi}} G(\ol{m}_2), 
\end{split}
\]
where the fibering maps $\psi_i : G(m_i) \twoheadrightarrow \Gamma$ surject onto a non-abelian group $\Gamma$ and $\ol{\psi}_i : G(\ol{m}_i) \twoheadrightarrow \ol{\Gamma}$ surject onto the corresponding quotient $\ol{\Gamma}$ of $\Gamma$ as described above in Lemma \ref{onesideprojectionlemma}.  Let $\varpi : \Gamma \twoheadrightarrow \ol{\Gamma}$ denote the canonical surjection. Here and throughout this section, we let $\pi_i : G(m_i) \twoheadrightarrow G(\ol{m}_i)$ denote the reduction modulo $\ol{m}_i$ map restricted to $G(m_i)$ and 
\begin{equation} \label{defofNandK}
N^G(m_i) := \ker \psi_i, \quad\quad N^{\ol{G}}(\ol{m}_i) := \ker \ol{\psi}_i.
\end{equation}
To view things more globally, we define the open subgroups $N^G \subseteq G$ and $N^{\ol{G}} \subseteq \ol{G}$ by
\[
N^G := \pi_{\GL_2}^{-1}\left( N^G(m_1) \times N^G(m_2) \right), \quad\quad N^{\ol{G}} := \pi_{\GL_2}^{-1}\left( N^{\ol{G}}(\ol{m}_1) \times N^{\ol{G}}(\ol{m}_2) \right)
\]
and the maps 
\[
\begin{split}
\psi &: G \longrightarrow \Gamma, \\
\ol{\psi} &: \ol{G} \longrightarrow \ol{\Gamma}
\end{split}
\]
by $\psi(g) := \psi_1(g \mod m_1) = \psi_2(g \mod m_2)$ and $\ol{\psi}(g) := \ol{\psi}_1(g \mod \ol{m}_1) = \ol{\psi}_2(g \mod \ol{m}_2)$.  Then $N^G = \ker \psi$, $N^{\ol{G}} = \ker \ol{\psi}$, and we have a commuting diagram of exact sequences
\begin{equation} \label{Gammadiagram}
\begin{tikzcd}
1 \rar & N^G \rar \dar & G \rar{\psi} \dar & \Gamma \rar \dar{\varpi} & 1 \\
1 \rar & N^{\ol{G}} \rar & \ol{G} \rar{\ol{\psi}} & \ol{\Gamma} \rar & 1
\end{tikzcd}
\end{equation}
in which all unlabeled arrows denote either inclusion maps or trivial surjections.  For any $n \in \mbn$, we may now consider the subgroups $N^G(n) \subseteq G(n) \subseteq \GL_2(\mbz/n\mbz)$ and $N^{\ol{G}}(n) \subseteq \ol{G}(n) \subseteq \GL_2(\mbz/n\mbz)$; we note that $N^G(n) \subseteq N^{\ol{G}}(n)$ and caution the reader that this containment may be proper, especially when $n = \ol{m}$.  Since we are considering the genus of $G$, we are interested in its intersection with $\SL_2(\hat{\mbz})$.
Here and throughout the rest of this section, we will employ the following notation:  for any open subgroup $H \subseteq \GL_2(\hat{\mbz})$, we define
\begin{equation} \label{defofHsubSL2}
H_{\SL_2} := H \cap \SL_2(\hat{\mbz}).
\end{equation}
Note that, for any $n \in \mbn$,
\[
H_{\SL_2}(n) \subseteq H(n) \cap \SL_2(\mbz/n\mbz).
\]
We caution the reader that this containment may be proper when $n$ isn't a multiple of $m_{\GL_2}(H)$. 
The analogue of \eqref{Gammadiagram} obtained after intersecting with $\SL_2(\hat{\mbz})$ is
\begin{equation} \label{SL2Gammadiagram}
\begin{tikzcd}
1 \rar & N^G_{\SL_2} \rar \dar & G_{\SL_2} \rar{\psi\vert_{G_{\SL2}}} \dar & \Gamma_{\SL_2} \rar \dar{\varpi\vert_{\Gamma_{\SL2}}} & 1 \\
1 \rar & N^{\ol{G}}_{\SL_2} \rar & \ol{G}_{\SL_2} \rar{\ol{\psi}\vert_{\ol{G}_{\SL2}}} & \ol{\Gamma}_{\SL_2} \rar & 1,
\end{tikzcd}
\end{equation}
where $\Gamma_{\SL_2} := \psi_1\left( G_{\SL_2}(m_1) \right) \cap \psi_2\left( G_{\SL_2}(m_2) \right)$ and $\ol{\Gamma}_{\SL_2} := \ol{\psi}_1\left( \ol{G}_{\SL_2}(\ol{m}_1) \right) \cap \ol{\psi}_2\left( \ol{G}_{\SL_2}(\ol{m}_2) \right)$ are as in Lemma \ref{SL2contractionlemma} (actually, by the definition \eqref{defofHsubSL2}, we in fact have $\Gamma_{\SL_2} := \psi_1\left( G_{\SL_2}(m_1) \right) = \psi_2\left( G_{\SL_2}(m_2) \right)$ and likewise with $\ol{\Gamma}_{\SL_2}$).  In what follows, our goal is to understand the image of $G_{\SL_2}$ inside $\ol{G}_{\SL_2}$, which is equivalent to understanding the image of $G_{\SL_2}(\ol{m})$ inside $\ol{G}_{\SL_2}(\ol{m})$.


\begin{lemma} \label{NsubSlemma}
Assume the notation outlined above (in particular, assume that $G$ is a non-abelian entanglement group with $\SL_2$-level dividing $\ol{m}$).   We have
\begin{equation} \label{kervarpicontainedinZofA}
\ker \varpi \subseteq Z(\Gamma);
\end{equation}
in particular, for each $i \in \{1, 2 \}$, $\psi_i\left( N^{\ol{G}}(m_i) \right) \subseteq Z(\Gamma)$.  Furthermore, we have
\begin{equation} \label{commutatorcontainmentsforNsubS}
\begin{split}
\left[G(\ol{m}_i),N^{\ol{G}}(\ol{m}_i)\right] &\subseteq N^G_{\SL_2}(\ol{m}_i), \\ 
\left[G(\ol{m}_i),G(\ol{m}_i)\right] &\nsubseteq N^G_{\SL_2}(\ol{m}_i).
\end{split}
\end{equation}
\end{lemma}
\begin{proof}
Since $m_{\SL_2}(G)$ divides $\ol{m}$, we have that, for each $i \in \{1, 2 \}$, $\psi_{i}( \ker \pi_{i} \cap \SL_2(\mbz/m_{i}\mbz) ) = 1_\Gamma$, and so $\psi_{i} |_{\ker \pi_{i}}$ factors through the determinant map.  It follows that, for each $g \in G(m_{i})$ and $k \in \ker \pi_{i}$, we have
\[
\psi_{i}(g k g^{-1}) = \psi_{i}(k),
\]
and by surjectivity of $\psi_{i}$ we thus see that $\psi_i(\ker \pi_i)$ is contained in the center of $\Gamma$.  Since $\ker \varpi = \psi_1(\ker \pi_1) \psi_2(\ker \pi_2)$, this establishes \eqref{kervarpicontainedinZofA}, and it follows from $\varpi \circ \psi_i = \ol{\psi}_i \circ \pi_i$ that $\psi_i\left( N^{\ol{G}}(m_i) \right) \subseteq Z(\Gamma)$.  The first containment in \eqref{commutatorcontainmentsforNsubS} follows from this by further considering the isomorphism $G(m_i) / N^G(m_i) \simeq \Gamma$ and then projecting modulo $\ol{m}_i$.  To see why $[G(\ol{m}_i),G(\ol{m}_i)] \nsubseteq N^G_{\SL_2}(\ol{m}_i)$, suppose on the contrary that $[G(\ol{m}_i),G(\ol{m}_i)] \subseteq N^G_{\SL_2}(\ol{m}_i)$.  Since $N^G_{\SL_2}(m_i) = \pi_{\SL_2}^{-1}(N^G_{\SL_2}(\ol{m}_i))$, we then see that
\[
[G(m_i),G(m_i)] \subseteq \pi_{\SL_2}^{-1}([G(\ol{m}_i),G(\ol{m}_i)]) \subseteq N^G_{\SL_2}(m_i) \subseteq N^G(m_i),
\]
contradicting the fact that $\Gamma \simeq G(m_i) / N^G(m_i)$ is non-abelian.  This establishes that $[G(\ol{m}_i),G(\ol{m}_i)] \nsubseteq N^G_{\SL_2}(\ol{m}_i)$, finishing the proof.
\end{proof}
\begin{cor} \label{nontrivialabelianquotientcor}
Let $\ol{G}$ be a pre-twist group, let $G \in \mc{G}_{\nonab}$ be a twist cover of $\ol{G}$ and let $m$, $\ol{m}$ be as in \eqref{twistcoverconditions}.  Suppose that
$m = m_1 m_2$ is a permissible factorization for which $G(m) \simeq G(m_1) \times_{\psi} G(m_2)$ has a non-abelian common quotient $\Gamma$.  Then, defining $\ol{m}_i := \gcd(\ol{m},m_i)$ for $i \in \{1, 2 \}$, the common quotient $\ol{\Gamma}$ associated to $\ol{G}(\ol{m}) \simeq \ol{G}(\ol{m}_1) \times_{\ol{\psi}} \ol{G}(\ol{m}_2)$ satisfies $\ol{\Gamma} \neq \{ 1 \}$.  Consequently we have
\begin{equation} \label{thinneddownlevels}
G \in \mc{G}_{\nonab}(0) \; \Rightarrow \; m_{\SL_2}(G) \in \left\{ \begin{matrix} 6,10,12,14,15,18,20,21,22,24,26,28,30, \\ 36,40,42,48,50,52,54,56,60,64,72,96 \end{matrix} \right\}.
\end{equation}
If we further assume that $G \in \mc{G}_{\nonab}^{\max}$, then $\ol{\Gamma}$ is abelian.
\end{cor}
\begin{proof}
By Lemma \ref{onesideprojectionlemma} and \eqref{kervarpicontainedinZofA}, we have that
\[
\ker \varpi \subseteq Z(\Gamma) \neq \Gamma,
\]
since $\Gamma$ is non-abelian.  This establishes that $\ol{\Gamma} \neq \{ 1 \}$.  It follows that the factorization $\ol{m} = \ol{m}_1 \ol{m}_2$ must be permissible, and therefore $m_{\SL_2}(G) = \ol{m}$ must belong to the subset of those levels listed in Corollary \ref{levelincreaseboundcor} which admit permissible factorizations, leading to \eqref{thinneddownlevels}.
Finally, if $\ol{\Gamma}$ were non-abelian, then $\ol{G} \in \mc{G}_{\nonab}$ and $G \subsetneq \ol{G}$, contradicting the hypothesis that $G \in \mc{G}_{\nonab}^{\max}$.
\end{proof}
\begin{proposition} \label{pretwistscenarioproposition}
Let $\ol{G}$ be a pre-twist group, let $G$ be a twist cover of $\ol{G}$, and assume the notation set above.  Then there exists a pair of surjective group homomorphisms $\tilde{\psi}_i : G(\ol{m}_i) \twoheadrightarrow \tilde{\Gamma}_i$ onto non-abelian groups $\tilde{\Gamma}_i$ with the following properties:
\begin{enumerate}
\item There exist surjective group homomorphisms $\tilde{\varpi}_i :\tilde{\Gamma}_i \twoheadrightarrow \ol{\Gamma}$ with $\ker \tilde{\varpi}_i \subseteq Z(\tilde{\Gamma}_i)$ and
\begin{equation} \label{psibarequalsvarpicircpsitilde}
\ol{\psi}_i = \tilde{\varpi}_i \circ \tilde{\psi}_i.
\end{equation}
\item Defining $\Gamma_{\SL_2}$ to be the common value $\psi_1\left( G_{\SL_2}(m_1) \right) = \psi_2\left( G_{\SL_2}(m_2) \right)$ and $\tilde{\Gamma}_{i,\SL_2} := \tilde{\psi}_i\left(G_{\SL_2}(\ol{m}_i)\right)$, there are isomorphisms $\theta_i : \Gamma_{\SL_2} \rightarrow \tilde{\Gamma}_{i,\SL_2}$ satisfying $\tilde{\varpi}_i \vert_{\tilde{\Gamma}_{i,\SL_2}} \circ \theta_i = \varpi \vert_{\Gamma_{\SL_2}}$ and $\theta_i \circ \psi_i \vert_{G_{\SL_2}(m_i)} = \tilde{\psi}_i \vert_{G_{\SL_2}(\ol{m}_i)} \circ \pi_i \vert_{G_{\SL_2}(m_i)}$.
\end{enumerate}
Finally, under the isomorphism $G(m) \simeq G(m_1) \times_{\psi} G(m_2)$ we have that
\begin{equation} \label{criterioncheckableatlowerlevel}
G_{\SL_2}(m) \simeq \pi_{\SL_2}^{-1}\left( \tilde{\psi}_1|_{G_{\SL_2}(\ol{m}_1)}^{-1}(\theta_1(\Gamma_{\SL_2})) \times_{\theta^{-1} \circ \tilde{\psi}} \tilde{\psi}_2|_{G_{\SL_2}(\ol{m}_2)}^{-1}(\theta_2(\Gamma_{\SL_2})) \right).
\end{equation}
\end{proposition}

\begin{proof}
We define
\[
\tilde{\Gamma}_i := G(\ol{m}_i) / N^G_{\SL_2}(\ol{m}_i)
\] 
and let $\tilde{\psi}_i : G(\ol{m}_i) \twoheadrightarrow \tilde{\Gamma}_i$ be the canonical surjection.  By Lemma \ref{NsubSlemma}, $\tilde{\Gamma}_i$ is non-abelian.  By the definition \eqref{defofNandK}, we see that $N^G_{\SL_2}(\ol{m}_i) \subseteq N^{\ol{G}}(\ol{m}_i)$, so there is a natural map
\[
\tilde{\varpi}_i : \tilde{\Gamma}_i := G(\ol{m}_i)/N^G_{\SL_2}(\ol{m}_i) \longrightarrow G(\ol{m}_i)/N^{\ol{G}}(\ol{m}_i) \simeq \ol{\Gamma}.
\]
We note the commuting diagram
\begin{equation*} 
\begin{tikzcd}
G(m_i) \rar \arrow[black, bend left]{rrr}{\psi_i} \dar{\pi_i}  & G(m_i)/N^G_{\SL_2}(m_i) \rar  & G(m_i)/N^G(m_i) \dar \rar{\simeq} & \Gamma \dar{\varpi} \\
G(\ol{m}_i) \rar{\tilde{\psi}_i} & G(\ol{m}_i)/N^G_{\SL_2}(\ol{m}_i) \rar \arrow[black, bend right]{rr}{\tilde{\varpi}_i} & G(\ol{m}_i)/N^{\ol{G}}(\ol{m}_i) \rar{\simeq} & \ol{\Gamma},
\end{tikzcd}
\end{equation*}
which implies that $\varpi \circ \psi_i = \tilde{\varpi}_i \circ \tilde{\psi}_i \circ \pi_i$.  Using this together with \eqref{functionaleqnwithpsiandpsibar}, the functional equation \eqref{psibarequalsvarpicircpsitilde} is then established.  Futhermore, it follows from the first containment in \eqref{commutatorcontainmentsforNsubS} that $\ker \tilde{\varpi}_i \subseteq Z(\tilde{\Gamma}_i)$.

We now construct the maps $\theta_i$.  We have $\Gamma_{\SL_2} := \psi_i(G_{\SL_2}(m_i)) \simeq G_{\SL_2}(m_i)/N^G_{\SL_2}(m_i)$, and since $m_{\SL_2}(G) = \ol{m}$ and by definition of $N^G_{\SL_2}$, we have
\[
\begin{split}
G_{\SL_2}(m_i) \, = & \; \pi_{\SL_2}^{-1}\left( G_{\SL_2}(\ol{m}_i) \right) \\
N^G_{\SL_2}(m_i) \, =& \; \pi_{\SL_2}^{-1}\left( N^G_{\SL_2}(\ol{m}_i) \right),
\end{split}
\] 
and thus we see that the reduction modulo $\ol{m}_i$ map induces an isomorphism that defines  $\theta_i$:
\[
\theta_i : \Gamma_{\SL_2} \simeq \frac{G_{\SL_2}(m_i)}{N^G_{\SL_2}(m_i)} \rightarrow \frac{G_{\SL_2}(\ol{m}_i)}{N^G_{\SL_2}( \ol{m}_i )} \simeq \tilde{\psi}_i\left( G_{\SL_2}(\ol{m}_i) \right) = \tilde{\Gamma}_{i,\SL_2}.
\]
Furthermore, the commuting diagram
\begin{equation} \label{smallerfancydiagram}
\begin{tikzcd}
G_{\SL_2}(m_i) \rar \arrow[black, bend left]{rr}{\psi_i} \dar{\pi_i}  & G_{\SL_2}(m_i)/N^G_{\SL_2}(m_i) \dar \rar{\simeq} & \Gamma_{\SL_2} \dar{\theta_i} \arrow{rd}{\varpi} & \\
G_{\SL_2}(\ol{m}_i) \rar \arrow[black, bend right]{rr}{\tilde{\psi}_i} & G_{\SL_2}(\ol{m}_i)/N^G_{\SL_2}(\ol{m}_i) \rar{\simeq} & \tilde{\Gamma}_{i,\SL_2} \rar{\tilde{\varpi}_i}  & \ol{\Gamma}_{\SL_2} \simeq \ol{G}_{\SL_2}(\ol{m}_i)/N^{\ol{G}}_{\SL_2}(\ol{m}_i),
\end{tikzcd}
\end{equation}
illustrates that $\tilde{\varpi}_i \vert_{\tilde{\Gamma}_{i,\SL_2}} \circ \theta_i = \varpi \vert_{\Gamma_{\SL_2}}$ and
\begin{equation} \label{functionidentity}
\theta_i \circ \psi_i \vert_{G_{\SL_2}(m_i)} = \tilde{\psi}_i \vert_{G_{\SL_2}(\ol{m}_i)} \circ \pi_i \vert_{G_{\SL_2}(m_i)}. 
\end{equation}
Finally, Lemma \ref{SL2contractionlemma}, together with \eqref{functionidentity} and the fact that $G_{\SL_2}(m) = \pi_{\SL_2}^{-1}\left( G_{\SL_2}(\ol{m}) \right)$, imply that
\[
\begin{split}
G_{\SL_2}(m) &\simeq \psi_1 \vert_{G_{\SL_2}(m_1)}^{-1} \left( \Gamma_{\SL_2} \right) \times_\psi \psi_2 \vert_{G_{\SL_2}(m_2)}^{-1} \left( \Gamma_{\SL_2} \right) \\
&= \pi_1 \vert_{G_{\SL_2}(m_1)}^{-1} \tilde{\psi}_1 \vert_{G_{\SL_2}(\ol{m}_1)}^{-1} \theta_1 \left( \Gamma_{\SL_2} \right) \times_{\theta^{-1} \tilde{\psi} \pi} \pi_2 \vert_{G_{\SL_2}(m_2)}^{-1} \tilde{\psi}_2 \vert_{G_{\SL_2}(\ol{m}_2)}^{-1} \theta_2 \left( \Gamma_{\SL_2} \right) \\
&= \pi_{\SL_2}^{-1}\left( \tilde{\psi}_1\vert_{G_{\SL_2}(\ol{m}_1)}^{-1}(\theta_1(\Gamma_{\SL_2})) \times_{\theta^{-1} \tilde{\psi}} \tilde{\psi}_2
\vert_{G_{\SL_2}(\ol{m}_2)}^{-1}(\theta_2(\Gamma_{\SL_2})) \right).
\end{split}
\]
\end{proof}

The main point of Proposition \ref{pretwistscenarioproposition} is that the right-hand side of \eqref{criterioncheckableatlowerlevel} involves information just from level $\ol{m}$, excepting only the subgroups $\tilde{H}_i := G_{\SL_2}(\ol{m}_i) \subseteq \ol{G}_{\SL_2}(\ol{m}_i)$, which 
satisfy the condition
\begin{equation} \label{keyconditionontheHs}
\tilde{\psi}_1\left( \tilde{H}_1 \right) \simeq \tilde{\psi}_2\left( \tilde{H}_2 \right).
\end{equation}
In our search for pre-twist groups, we can thus take $\tilde{H}_i \subseteq \ol{G}_{\SL_2}(\ol{m}_i)$ to be arbitrary subgroups that happen to satisfy \eqref{keyconditionontheHs}. Thus, we have derived necessary conditions for $\ol{G}$ to be a pre-twist group, which can be checked from $\ol{G}$ alone (at level $\ol{m}$). We emphasize this point in the following corollary, which is then used to prove Proposition \ref{keypropositionformainthm}.

\begin{cor} \label{keycorollarytobeused}
Let $\ol{G}$ be a pre-twist group, let $G$ be a twist cover of $\ol{G}$ and let 
\[
m := m_{\GL_2}(G), \quad \ol{m} := m_{\GL_2}(\ol{G}).  
\]
Suppose that
$m = m_1 m_2$ is a permissible factorization for which $G(m) \simeq G(m_1) \times_{\psi} G(m_2)$ has a non-abelian common quotient $\Gamma$.  For $i\in\{1,2\}$, define $\ol{m}_i := \gcd(\ol{m},m_i)$, write
\[
\ol{G}(\ol{m}) \simeq \ol{G}(\ol{m}_1) \times_{\ol{\psi}} \ol{G}(\ol{m}_2),
\]
and set $N^{\ol{G}}(\ol{m}_i)=\ker\ol{\psi}_i$.
Then there exist subgroups $\tilde{N}_i\subseteq N^{\ol{G}}(\ol{m}_i)\cap\SL_2(\Z/\ol{m}_i\Z)$ (with $\tilde{N}_i \neq N^{\ol{G}}(\ol{m}_i)\cap\SL_2(\Z/\ol{m}_i\Z)$ in case $G \in \mc{G}_{\nonab}^{\max}$), with each $\tilde{N}_i$ normal in $\ol{G}(\ol{m}_i)$, 
satisfying
\begin{equation} \label{keypropertyofNsubitilde}
[\ol{G}(\ol{m}_i),N^{\ol{G}}(\ol{m}_i)] \subseteq \tilde{N}_i \nsupseteq [\ol{G}(\ol{m}_i),\ol{G}(\ol{m}_i)].
\end{equation}
Furthermore, setting $\tilde{\psi}_i : \ol{G}(\ol{m}_i) \longrightarrow \ol{G}(\ol{m}_i)/\tilde{N}_i$, there exist subgroups $\tilde{H}_i \subseteq \ol{G}_{\SL_2}(\ol{m}_i)$ and isomorphisms $\theta_i : B \rightarrow \tilde{\psi}_i\left( \tilde{H}_i \right)$ (for some group $B$) satisfying \begin{equation} \label{respectsthefibering}
\forall b \in B, \quad \tilde{\varpi}_1 \left( \theta_1 (b) \right) = \tilde{\varpi}_2 \left( \theta_2 (b) \right)
\end{equation}
and such that, if
\[
S := \tilde{\psi}_1 |_{\tilde{H}_1}^{-1} \left( \theta_1(B) \right) \times_{\theta^{-1} \circ \tilde{\psi}} \tilde{\psi}_2 |_{\tilde{H}_2}^{-1} \left( \theta_2(B) \right),
\]
then the modular curve $X_{\tilde{G}}$ is isomorphic over $\ol{\mbq}$ to the modular curve $X_{\tilde{S}}$.  In particular, there are embeddings $\theta_i : B \hookrightarrow \tilde{\psi}_i\left( \ol{G}_{\SL_2}(\ol{m}_i) \right)$ satisfying \eqref{respectsthefibering} and such that, if
\begin{equation} \label{defofSprime}
S' := \tilde{\psi}_1 |_{\ol{G}_{\SL_2}(\ol{m}_1)}^{-1} \left( \theta_1(B) \right) \times_{\theta^{-1} \circ \tilde{\psi}} \tilde{\psi}_2 |_{\ol{G}_{\SL_2}(\ol{m}_1)}^{-1} \left( \theta_2(B) \right),
\end{equation}
then $X_{\tilde{G}}$ is a geometric cover of $X_{\tilde{S}'}$, which in turn is a geometric cover of $X_{\tilde{\ol{G}}}$.
\end{cor}
\begin{proof}
This is essentially a direct translation of Proposition \ref{pretwistscenarioproposition}, taking $\tilde{N}_i = N^G_{\SL_2}(\ol{m}_i)$, $B=\Gamma_{\SL_2}$, and $\tilde{H}_i = G_{\SL_2}(\ol{m}_i)$. It is straightforward to verify that $\tilde{N}_i \unlhd \ol{G}(\ol{m}_i)$ and $\tilde{N}_i\subseteq N^{\ol{G}}(\ol{m}_i)\cap \SL_2(\Z/\ol{m}_i\Z)$. The fact that $[\ol{G}(\ol{m}_i),N^{\ol{G}}(\ol{m}_i)]\subseteq \tilde{N}_i$ and $[\ol{G}(\ol{m}_i),\ol{G}(\ol{m}_i)]\nsubseteq\tilde{N}_i$ can be seen directly from Lemma \ref{NsubSlemma}.  Finally, in case $G \in \mc{G}_{\nonab}^{\max}$, we have $[\ol{G}(\ol{m}_i),\ol{G}(\ol{m}_i)] \subseteq N^{\ol{G}}(\ol{m}_i) \cap \SL_2(\mbz/\ol{m}_i\mbz)$, which by \eqref{keypropertyofNsubitilde} forces $\tilde{N}_i$ to be a proper subgroup of $N^{\ol{G}}(\ol{m}_i) \cap \SL_2(\mbz/\ol{m}_i\mbz)$.

Regarding the modular curve $X_{\tilde{S}'}$, it is straightforward to see that $S \subseteq S'$, so it follows immediately from $X_{\tilde{G}} \simeq_{\ol{\mbq}} X_{\tilde{S}}$ that $X_{\tilde{G}}$ is a geometric cover of $X_{\tilde{S}'}$.  Finally, we claim that $S' \subseteq \ol{G}_{\SL_2}(\ol{m})$.  Indeed, if $(s_1,s_2) \in S'$, then for each $i \in \{1, 2\}$, we have $\tilde{\psi}_i(s_i) = \theta(b)$ for some (fixed) $b \in B$.  Now using \eqref{psibarequalsvarpicircpsitilde} together with \eqref{respectsthefibering}, we find that
\[
\ol{\psi}_1(s_1) = \tilde{\varpi}_1\left( \tilde{\psi}_1(s_1) \right) = \tilde{\varpi}_1\left( \theta_1(b) \right) = \tilde{\varpi}_2 \left( \theta_2(b) \right) = \tilde{\varpi}_2\left( \tilde{\psi}_2(s_2) \right) = \ol{\psi}_2(s_2).
\]
Thus $(s_1,s_2) \in \ol{G}_{\SL_2}(\ol{m})$, which establishes that $X_{\tilde{S}'}$ is a geometric cover of $X_{\tilde{\ol{G}}}$, finishing the proof.
\end{proof}

\begin{remark} \label{keyremarktobeused}
We included the group $S' \supseteq S$ in the statement of Corollary \ref{keycorollarytobeused} since it somewhat simplifies our computer search.  Indeed, since $\genus(X_{\tilde{G}}) \geq \genus(X_{\tilde{S}'})$, if for a given $\ol{G}$ our search produces no $S'$ as in \eqref{defofSprime} with $\genus(X_{\tilde{S}'}) = 0$, then it follows that there are no twist covers $G \in \mc{G}_{\nonab}^{\max}$ of $\ol{G}$ with $\genus(X_{\tilde{G}}) = 0$. 
\end{remark}

\subsection{An algorithm to search for pre-twist groups with maximal twist covers}
We now describe the algorithm used to search for pre-twist groups of genus zero that have maximal twist covers of genus zero. \\

\noindent \textbf{Step 1.}  For a fixed level
\begin{equation} \label{listofmbars}
\ol{m} \in \left\{ \begin{matrix} 6,10,12,14,15,18,20,21,22,24,26,28,30, \\ 36,40,42,48,50,52,54,56,60,64,72,96 \end{matrix} \right\},
\end{equation}
we construct (as a list) the set $\mc{G}^{m_{\GL_2}=\ol{m}}(0)$ of open subgroups $\ol{G} \subseteq \GL_2(\hat{\mbz})$ of genus zero and $\GL_2$-level $\ol{m}$ (see Corollary \ref{nontrivialabelianquotientcor} and Definition \ref{pretwistgroupdefinition}). \\

\noindent \textbf{Step 2.}  For each permissible factorization $\ol{m} = \ol{m}_1 \ol{m}_2$, we construct the subset $\mc{G}_{\ab}^{m_{\GL_2}=\ol{m}}\left(0,(\ol{m}_1,\ol{m}_2) \right)$ of all $\ol{G} \in \mc{G}^{m_{\GL_2}=\ol{m}}(0)$ with the property that, under $\ol{G}(\ol{m}) \subseteq \ol{G}(\ol{m}_1) \times \ol{G}(\ol{m}_2)$, the common quotient in the fibered product associated to $\ol{G}(\ol{m})$ via Lemma \ref{goursatlemma} is a non-trivial abelian group (see Corollary \ref{nontrivialabelianquotientcor}). \\

\noindent \textbf{Step 3.}  For each $\ol{G} \in \mc{G}_{\ab}^{m_{\GL_2}=\ol{m}}\left(0,(\ol{m}_1,\ol{m}_2) \right)$, denoting by $\ol{\psi} = (\ol{\psi}_1,\ol{\psi}_2)$ the pair of surjective group homomorphisms implicit in the fibered product $\ol{G}(\ol{m}) \simeq \ol{G}(\ol{m}_1) \times_{\ol{\psi}} \ol{G}(\ol{m}_2)$ and by $N^{\ol{G}}(\ol{m}_i) := \ker \ol{\psi}_i \subseteq \ol{G}(\ol{m}_i)$, we search for normal subgroups $\tilde{N}_i \unlhd \ol{G}(\ol{m}_i)$ satisfying $\tilde{N}_i \subsetneq N^{\ol{G}}(\ol{m}_i) \cap \SL_2(\mbz/\ol{m}_i\mbz)$ and the property \eqref{keypropertyofNsubitilde}.  We create a new list $\mc{G}_{\ab,\; \pot}^{m_{\GL_2}=\ol{m}}\left( 0,(\ol{m}_1,\ol{m}_2) \right)$ of \emph{potential} pre-twist groups, consisting of the triples $(\ol{G}(\ol{m}), \tilde{N}_1, \tilde{N}_2)$ found by this search.  Note that a given group $\ol{G}(\ol{m})$ may belong to more than one triple in this list. \\

\noindent \textbf{Step 4.} For each triple $(\ol{G}(\ol{m}), \tilde{N}_1, \tilde{N}_2) \in \mc{G}_{\ab,\; \pot}^{m_{\GL_2}=\ol{m}}\left( 0,(\ol{m}_1,\ol{m}_2) \right)$, denoting by $\tilde{\Gamma}_i := \ol{G}(\ol{m}_i)/\tilde{N}_i$ and by $\tilde{\psi}_i : \ol{G}(\ol{m}_i) \twoheadrightarrow \tilde{\Gamma}_i$ the canonical projection, we search for finite groups $B$ together with embeddings $\theta_i : B \hookrightarrow \tilde{\psi}_i\left( \ol{G}_{\SL_2}(\ol{m}_i) \right)$ satisfying \eqref{respectsthefibering}.  For each such pair $\left( B,\theta = (\theta_1,\theta_2) \right)$, we form the fibered product
\[
S' := \tilde{\psi}_1 |_{\ol{G}_{\SL_2}(\ol{m}_1)}^{-1} \left( \theta_1(B) \right) \times_{\theta^{-1} \tilde{\psi}} \tilde{\psi}_2 |_{\ol{G}_{\SL_2}(\ol{m}_2)}^{-1} \left( \theta_2(B) \right)
\]
and form a new \emph{final} list $\mc{G}_{\ab,\; \fin}^{m_{\GL_2}=\ol{m}}\left( 0,(\ol{m}_1,\ol{m}_2) \right)$ consisting of those quadruples $(\ol{G}(\ol{m}),\tilde{N}_1,\tilde{N}_2,S')$ for which 
the genus of $X_{\tilde{S}'}$ is zero (see Corollary \ref{keycorollarytobeused} and Remark \ref{keyremarktobeused}). \\
\begin{remark}
When computing the initial list of genus zero subgroups $G \subseteq \GL_2(\hat{\mbz})$ in Step $1$, we make use of the following memory-saving measures:
\begin{enumerate}
    \item For any level $\ol{m}$ that does not appear in the list \eqref{cumminspaulilistoflevels}, by Lemma \ref{levelincreaseboundcor}, any $G$ of genus zero and $\GL_2$-level $\ol{m}$ must satisfy $-I \notin G$ and $m_{\GL_2}(\tilde{G}) = \ol{m}/2$.  We therefore first construct the list of subgroups $G_0$ of $\GL_2$-level $\ol{m}/2$ satisfying $-I \in G_0$ and then, for each such $G_0$, search for index two subgroups $G \subseteq G_0$ of $\GL_2$-level $\ol{m}$ with $-I \notin G$.
    \item Searching directly among all subgroups of $\GL_2$-level $48$ is memory-intensive enough to be prohibitively difficult on most machines.  To work around this problem, we instead loaded separately the list of all genus zero subgroups of level $3$ and the list of all genus zero subgroups of level $16$, and then constructed directly every possible fibered product between the groups arising in those two lists.
\end{enumerate}
\end{remark}
\noindent \emph{Proof of Proposition \ref{keypropositionformainthm}.}
The above algorithm was implemented on a computer, using the MAGMA computational algebra system\footnote{MAGMA code for this search has been made available on the arXiv.} \cite{magma}.  For each level $\ol{m}$ from Step 1 and permissible factorization $\ol{m} = \ol{m}_1 \ol{m}_2$, the search concluded that $\mc{G}_{\ab,\; \fin}^{m_{\GL_2}=\ol{m}}\left( 0,(\ol{m}_1,\ol{m}_2) \right) = \emptyset$.  By Remark \ref{keyremarktobeused}, we conclude that
\[
G \in \mc{G}_{\nonab}^{\max}(0) \; \Longrightarrow \; m_{\GL_2}(G) = m_{\SL_2}(G).
\]

Finally, the assertion \eqref{finitelistforkeyprop} follows from \eqref{thinneddownlevels}, together with a straightforward computer search that we also carried out using MAGMA.  This search also yielded the data in Tables \ref{tableforlevelsneq30}, \ref{tableforlevel30} and \ref{tableforlevel60} of Section \ref{introduction}. \hfill $\Box$ \\



\begin{remark}
The key takeaway from Proposition \ref{keypropositionformainthm} is that we have a finite list of $\GL_2$-levels to consider when searching for maximal genus $0$ non-abelian entanglement groups $G$ (each necessarily satisfying $m_{\GL_2}(G)=m_{\SL_2}(G)$, by the proposition). This second search is then what establishes \eqref{listofmaximalgroups} in Theorem \ref{mainthm}.
\end{remark}

\bigskip


\section{Explicit Models for Modular Curves}
\label{Section:ExplicitModels}

In the previous section, we proved the first part of Theorem \ref{mainthm}, which is that (up to conjugation) there are exactly four maximal non-abelian entanglement groups, $G \in\left\{G_6,G_{10},G_{15},G_{18}\right\}$, for which the associated modular curve $X_{\tilde G}$ has genus $0$.
In all four cases the underlying entanglement is $S_3$, and in all four cases $-I\in G$ so that $\tilde{G}=G$.
Three of the curves are defined over $\Q$, while $X_{G_{15}}$ is defined over $\Q(\sqrt{-15})$.
In this section, we complete the proof of the theorem by determining explicit equations for the modular curves.
More precisely, we determine a rational parameter $t$ on each $X_{G}$ as well as an explicit formula for $j_G(t)$. Work for one of the curves, $X_{G_6}$, is omitted, as this curve was previously studied in \cite{braujones}.

Our approach to finding the explicit models is essentially one of ``gluing'' along the common non-abelian quotient $\Gamma$, in the decomposition of $G=G_m$ into the fiber product $G(m)=G(m_1)\times_{\psi} G(m_2)$.
This process is described {\em in general} in Section \ref{Section:GeneralYoga}, which will also serve as a foundation for future work (when $g>0$).
However, for each of the curves being considered here, the specific underlying entanglement group is $\Gamma\cong S_3$.
Hence, we prove in Section \ref{Section:S3Yoga} a lemma that explicitly describes the gluing mechanism in that special case.
Once the computational framework is fully developed in principle, it is then implemented in each of the three cases using SageMath \cite{sage}.

Since the groups $G(m_1)$ and $G(m_2)$ play such a crucial role in our analysis, this data is collected below in Table \ref{Table:CPData}. Note that $\borel(3)$ refers to the Borel group at $3$, while $\XG$ refers to the unique index $5$ subgroup of $\GL_2\left(\Z/5\Z\right)$ containing $\mc{N}_s(5)$ (the normalizer of split Cartan). We also include for reference the usual modular curve data vector $(d,c_2,c_3,c_{\infty})$ in each case, as well as the Cummins-Pauli label for the curve. 

{
\renewcommand{\arraystretch}{1.25}
\begin{table}[!ht]\label{Table:CPData}
\begin{center}
\begin{tabular}{ c | c | c | c }
$m$ & $G(m_1),G(m_2)$ & $(d,c_2,c_3,c_{\infty})$ & \text{C-P Label}\\ \hline
$6$ & $\GL_2(\Z/2\Z),\GL_2(\Z/3\Z)$ & $(6,0,3,1)$ & $6A^0$\\
$10$ & $\GL_2(\Z/2\Z),\XG$ & $(30,0,6,3)$ & $10E^0$ \\
$15$ & $\GL_2(\Z/3\Z),\XG$ &  $(15,3,3,1)$ & $15A^0$ \\
$18$ & $\GL_2(\Z/2\Z),\pi_9^{-1}\left(\borel(3)\right)$ & $(24,0,3,4)$ & $18C^0$
\end{tabular}
\end{center}\caption{Maximal Genus $0$ Non-Abelian Entanglement Curves}
\end{table}
}

\subsection{General Entanglement Curve Yoga}\label{Section:GeneralYoga}

Fix a non-abelian entanglement scenario, i.e., two subgroups, $G(m_1)\subseteq \GL_2\left(\Z/m_1\Z\right)$ and $G(m_2)\subseteq \GL_2\left(\Z/m_2\Z\right)$ (where $(m_1,m_2)=1$), which surject onto a common non-abelian quotient $\Gamma$ with kernels $N(m_1)$ and $N(m_2)$. For simplicity, assume that $-I$ is contained in each $N(m_i)$. We say that an elliptic curve $E/K$ has an entanglement of type $(G(m_1),N(m_1),G(m_2),N(m_2))$ if bases for $E[m_1]$ and $E[m_2]$ may be chosen over $\ol{K}$ such that\\
\\
\indent (1) $\Gal(K(E[m_1])/K)\cong G(m_1)$\\
\indent (2) $\Gal(K(E[m_2])/K)\cong G(m_2)$ and \\
\indent (3) $K(E[m_1])\cap K(E[m_2])=K(E[m_1])^{N(m_1)}=K(E[m_2])^{N(m_2)}$.\\
\\
The isomorphisms in (1) and (2) are induced by the isomorphisms of $\Aut(E[m_i])$ with $\GL_2(\Z/m_i\Z)$ that are determined by the choice of bases. Then the fixed fields in (3) are defined via those isomorphisms. Whenever the kernels $N(m_1)$ and $N(m_2)$ are uniquely determined by $G(m_1)$, $G(m_2)$ and $\Gamma$, we say that $E/K$ has an entanglement of type $(G(m_1),G(m_2),\Gamma)$ and simplify (3) to the following equivalent condition.\\
\\
\indent (3') $\Gal(K(E[m_1])\cap K(E[m_2])/K)\cong \Gamma$\\

In this section, we develop a method for determining explicit equations for a finite set of modular curves whose $K$-rational points ``correspond generically'' to elliptic curves $E/K$ that have an entanglement of type $(G(m_1),N(m_1),G(m_2),N(m_2))$, meaning that every such elliptic curve must correspond to a $K$-rational point on one of the modular curves. More precisely, after an appropriate choice of basis for $E[m_1m_2]$ over $\ol{K}$ we have
$$\Gal(K(E[m_1m_2])/K)\subseteq G(m_1)\times_{\psi}G(m_2)$$
for some $\psi_i:G(m_i)\twoheadrightarrow\Gamma$ with $\ker\psi_i=N(m_i)$ if and only if $j(E)$ lifts to a $K$-rational point on one of the modular curves.

The first step is to find an explicit model for the ``full product'' curve, the modular curve $X:=X_{G(m_1),G(m_2)}$ whose $K$-rational points correspond generically to elliptic curves $E/K$ that satisfy properties (1) and (2) from above. Since $-I$ is contained in both groups, $X$ can be obtained by crossing the modular curves $X_{G(m_1)}$ and $X_{G(m_2)}$ over the $j$-line. Next, we determine explicit models for the modular curves, $Y_{\Gamma,i}$ ($i=1,2$), which lie over $X$ and whose $K$-rational points correspond generically to elliptic curves $E/K$ for which $\Gal(K(E[m_i])/K)=N(m_i)$. Then each $Y_{\Gamma,i}$ is a Galois cover of $X$, whose Galois group, i.e., the Galois group of the corresponding extension of function fields, is isomorphic to $\Gamma$.
For each choice of isomorphisms, $\sigma_i:\Gal(Y_{\Gamma,i}/X)\to \Gamma$, it makes sense to form the diagonal quotient $\cX_{\sigma_1,\sigma_2}$ of the fiber product of $Y_{\Gamma,1}$ and $Y_{\Gamma,2}$ over $X$.
$$\cX_{\sigma_1,\sigma_2}:=Y_{\Gamma,1}\times_{X} Y_{\Gamma,2}/\left\{(\sigma_1^{-1}(g),\sigma_2^{-1}(g))|g\in \Gamma\right\}$$
\begin{Theorem}\label{Thm:GeneralYoga1}
Let $P$ be a $K$-rational point on $X_{G(m_1),G(m_2)}$, corresponding to an elliptic curve $E/K$ satisfying properties (1) and (2) from above. Then $E$ satisfies condition (3) if and only if $P$ lifts to a $K$-rational point on some $\cX_{\sigma_1,\sigma_2}$.
\end{Theorem}
\begin{proof}
First suppose that $E/K$ satisfies property (3) from above. Then for some Galois extension $L/K$ there are injections, $\alpha_i:L\hookrightarrow K(E[m_i])$ (over $K$), which identify $L$ with the fixed field of $N(m_i)$. But this fixed field is precisely the specialization of the function field of $Y_{\Gamma,i}$ to $P$. Thus, $\alpha_i$ induces an $L$-valued point, $\hat{\alpha}_i: \text{Spec}(L)\to Y_{\Gamma,i}$, which restricts to an isomorphism (over $K$) on the fiber over $P$. Identifying $\Gal(L/K)$ with $\Gamma$, we define $\sigma_1$ and $\sigma_2$ as follows. For $\tau\in \Gal(Y_{\Gamma,i}/X)$, we set $\sigma_i(\tau)=\alpha_i^{-1}\tau\alpha_i$. Consider the $L$-valued point of $\cX_{\sigma_1,\sigma_2}$ given by $\hat{P}=(\hat{\alpha}_1,\hat{\alpha}_2)$, which clearly lies over $P$. For any $g\in \Gal(L/K)$, we have $g(\hat{P})=(\hat{\alpha}_1\hat{g},\hat{\alpha}_2\hat{g})$, where $\hat{g}$ is the induced automorphism on $\text{Spec}(L)$ over $K$. On the other hand, if we act on $\hat{P}$ geometrically by $(\sigma_1^{-1}(g),\sigma_2^{-1}(g))$, we get
$$(\hat{\alpha}_1\hat{g}\hat{\alpha}_1^{-1}\hat{\alpha}_1,\hat{\alpha}_2\hat{g}\hat{\alpha}_2^{-1}\hat{\alpha}_2)=(\hat{\alpha}_1\hat{g},\hat{\alpha}_2\hat{g})=g(\hat{P}).$$
Thus, $\hat{P}$ is actually fixed by $\Gal(L/K)$ and hence $K$-rational.

Conversely, suppose $P$ lifts to a $K$-rational point $\hat{P}$ on some $\cX_{\sigma_1,\sigma_2}$. The key observation in this direction is that regardless of the choice of $(\sigma_1,\sigma_2)$, no nontrivial diagonal element $(g,g)$ fixes either $Y_{\Gamma,i}$. So the three function fields of $Y_{\Gamma,1}$, $Y_{\Gamma,2}$ and $\cX_{\sigma_1,\sigma_2}$ are all linearly disjoint inside the overall extension, and the compositum of any two contains the third.
$$Y_{\Gamma,1}\times_{X}\cX_{\sigma_1,\sigma_2} = Y_{\Gamma,1}\times_{X} Y_{\Gamma,2} = Y_{\Gamma,2}\times_{X}\cX_{\sigma_1,\sigma_2}$$
Specialization to $\hat{P}$ determines an isomorphism between the fibers of $Y_{\Gamma,1}$ and $Y_{\Gamma,2}$ over $P$. As noted above, this is equivalent to an isomorphism between the fixed fields of $K(E[m_1])$ and $K(E[m_2])$ by $N(m_1)$ and $N(m_2)$. So Condition (3) holds for $E/K$.
\end{proof}

While there are clearly $(\Aut\Gamma)^2$ choices for $(\sigma_1,\sigma_2)$, not all of the corresponding curves (i.e., function fields) of the form $\cX_{\sigma_1,\sigma_2}$ are distinct.
Moreover, the elements of the Galois group of $Y_{\Gamma,1}\times_{X} Y_{\Gamma,2}$ over $X$ may restrict to isomorphisms between some of these (distinct) intermediate fields. Therefore, it is not immediately clear from Theorem \ref{Thm:GeneralYoga1} how many distinct {\em modular curves} exist for each fixed entanglement type. The following theorem answers this question.

\begin{Theorem}\label{Thm:HowManyCurves}
There are $|\Aut\Gamma|$ distinct curves of the form $\cX_{\sigma_1,\sigma_2}$ lying over $X_{G(m_1),G(m_2)}$. However, each isomorphism class (over $X_{G(m_1),G(m_2)}$) is acted on faithfully by $\Inn\Gamma$. Hence, there are no more than $[\Aut\Gamma:\Inn\Gamma]$ modular curves for each fixed $(G(m_1),N(m_1),G(m_2),N(m_2))$ entanglement type.
\end{Theorem}
\begin{proof}
The group, $\Aut\Gamma\times\Aut\Gamma$ acts transitively on the set of curves, $\cX_{\sigma_1,\sigma_2}$, by post-composition on both sides.
$$(\tau_1,\tau_2):\cX_{\sigma_1,\sigma_2}\mapsto \cX_{\tau_1\sigma_1,\tau_2\sigma_2}$$
However, the diagonal subgroup acts trivially, since the group of geometric transformations by which the quotient of $Y_{\Gamma,1}\times_{X} Y_{\Gamma,2}$ is being taken remains the same. In fact, for any $\tau_1,\tau_2\in \Aut\Gamma\times\Aut\Gamma$ we have
$$\{(\sigma_1^{-1}(g),\sigma_2^{-1}(g))\}=\{((\tau_1\sigma_1)^{-1}(g),(\tau_2\sigma_2)^{-1}(g))\}\longleftrightarrow \tau_1=\tau_2.$$
So, the set $\{\cX_{\sigma_1,\sigma_2}\}$ actually only contains $|\Aut\Gamma|$ distinct curves, i.e, diagonal quotients of $Y_{\Gamma,1}\times_{X} Y_{\Gamma,2}$.

Now, fix $(\sigma_1,\sigma_2)$, which in turn fixes an isomorphism of $\Gal(Y_{\Gamma,1}\times_{X} Y_{\Gamma,2}/X)$ with $\Gamma\times \Gamma$. With this perspective, we may view the function field of $\cX_{\sigma_1,\sigma_2}$ as the fixed field of the diagonal subgroup, $\{(g,g)\}$. Moreover, any element $(g_1,g_2)$ then defines an isomorphism (via Galois) from this curve onto the one whose function field is fixed by $\{(g_1gg_1^{-1},g_2gg_2^{-1})\}$. It's easy to check that this curve is none other than $\cX_{\tau_1\sigma_1,\tau_2\sigma_2}$, where $\tau_1,\tau_2\in\Aut\Gamma$ are given by
$$\tau_1(g)=g_1^{-1} g g_1\qquad \tau_2(g)=g_2^{-1} g g_2.$$
So, when $\tau_1,\tau_2\in \Inn\Gamma$, the aforementioned action of $\Aut\Gamma\times\Aut\Gamma$ actually corresponds to an isomorphism between the two curves. Clearly, if we fix $\tau_2$ to be the identity, the resulting action of $\Inn\Gamma$ is faithful, which proves the theorem.
\end{proof}

\begin{remark}
Any specific choice of {\em maps}, $(\psi_1,\psi_2)$, determines exactly one of the above modular curves. We have specified only the kernels in the entanglement type in order to highlight the distinction and facilitate the counting of the modular curves. In addition, it is often more difficult in practice to nail down the maps than it is to specify the kernels.
\end{remark}

\subsection{$S_3$ Entanglement Modular Curves}\label{Section:S3Yoga}

The above construction can be made completely explicit in the case where $\Gamma=S_3$. Recall that the first step in the process is to determine the function field $L$ for the full product modular curve $X_{G(m_1),G(m_2)}$ by crossing the modular curves $X_{G(m_1)}$ and $X_{G(m_2)}$ over the $j$-line. Then, the $\psi$ maps on either side of the fiber product, $G(m_1)\times_{\psi} G(m_2)$, or more precisely their kernels, $N(m_1)$ and $N(m_2)$, will give rise to two $S_3$ extensions $L_1$ and $L_2$ of $L$. 
Without loss of generality, we may assume that these extensions are the splitting fields of two irreducible cubic polynomials over $L$, $x^3+Ax^2+Bx+C$ and $x^3+Dx^2+Ex+F$, whose roots in $\bar{L}$ are $\{s_1,s_2,s_3\}$ and $\{t_1,t_2,t_3\}$ (respectively). Identifying the Galois group of the compositum $L_1L_2$ over $L$ with $S_3\times S_3$, the function field for the entanglement modular curve will then be the subfield fixed by the diagonal subgroup. But this subfield is clearly generated over $L$ by the element $r:=s_1t_1+s_2t_2+s_3t_3$. Hence, an explicit equation for the $S_3$ entanglement modular curve, as an extension of $X_{G(m_1),G(m_2)}$, will be given by the minimal polynomial for $r$ over $L$.
The following lemma provides an explicit formula for that minimal polynomial.

\begin{lemma}\label{Lemma:DiagonalCubic}
Let $\{s_1,s_2,s_3\}$ and $\{t_1,t_2,t_3\}$ be the roots of the polynomials,
$x^3+Ax^2+Bx+C$ and $x^3+Dx^2+Ex+F$, respectively, in the compositum of the two splitting fields. In the same field, set
$$\delta=(s_1-s_2)(s_1-s_3)(s_2-s_3)(t_1-t_2)(t_1-t_3)(t_2-t_3)$$
and $r=s_1t_1+s_2t_2+s_3t_3$. (So, $\delta^2$ is the product of the two cubic discriminants.) Then $r$ is a root of the cubic, $x^3+Gx^2+Hx+I$, where
\begin{align*}
G&=-AD\\
H&=A^2E+D^2B-3BE\\
I&=-\tfrac{1}{2}\left(2CD^3+ABDE+2A^3F-9CDE-9ABF+27CF+\delta\right).
\end{align*}

\end{lemma}

\begin{proof}
This is easily verified by interpreting the coefficients as symmetric \mbox{functions} in the roots. 
$$
\begin{tabular}{ c | c  |c }
$A$ & $B$ & $C$\\ \hline
$-s_1-s_2-s_3$ & $s_1s_2+s_1s_3+s_2s_3$ & $-s_1s_2s_3$\\ 
& & \\ 
$D$ & $E$ & $F$\\ \hline
$-t_1-t_2-t_3$ & $t_1t_2+t_1t_3+t_2t_3$ & $-t_1t_2t_3$\\
& & \\
$G$ & $H$ & $I$\\ \hline
$-r_1-r_2-r_3$ & $r_1r_2+r_1r_3+r_2r_3$ & $-r_1r_2r_3$\\
\end{tabular}
$$
Take $r_1=r$, $r_2=s_1t_2+s_2t_3+s_3t_1$ and $r_3=s_1t_3+s_2t_1+s_3t_2$.
\end{proof}

\begin{remark}
It is irrelevant how we identify with $S_3$ on each side, i.e., which ``diagonal quotient'' we choose. Once the kernels of $\psi_1$ and $\psi_2$ are specified, there is only one entanglement modular curve up to isomorphism by Theorem \ref{Thm:HowManyCurves}, since  $[\Aut\Gamma:\Inn\Gamma]=1$ when $\Gamma=S_3$.
\end{remark}

\subsection{Level $10$}\label{Section:Level10}
Let $G(m_1)=\GL_2(\Z/2\Z)$, and let $G(m_2)=\XG\subseteq \GL_2(\Z/5\Z)$ be the unique index $5$ subgroup containing $\mc{N}_s(5)$ (the normalizer of split Cartan) as an index $3$ subgroup.  The group $\PGL_2(\mbz/5\mbz)$ contains as a subgroup an isomorphic copy of $S_4$, and $\XG$ may also be described as the full pre-image of that copy of $S_4$ under the canonical projection $\GL_2(\mbz/5\mbz) \twoheadrightarrow \PGL_2(\mbz/5\mbz)$ (it is often referred to as an \emph{exceptional} subgroup).
We may fix an isomorphism $\psi_1:G(m_1)\xrightarrow{\sim} S_3$ and a surjection $\psi_2:G(m_2)\twoheadrightarrow S_3$ whose kernel is contained in $\mc{N}_s(5)$ with index $2$.
Moreover, $G:=\pi_{\GL_2}^{-1}\left(G(m_1)\times_{\psi} G(m_2)\right)$ is conjugate to the non-abelian entanglement group $G_{10}$ in Theorem \ref{mainthm}. In this section we determine a parameter on the genus $0$ modular curve $\cX:=X_{G}$, as well as an explicit formula for the map from $\cX$ down to the $j$-line.


Closely following the general yoga of Section \ref{Section:GeneralYoga}, our first step is to find an explicit model for the full product modular curve.
However, since $G(m_1)$ is ``full,'' i.e., all of $\GL_2(\Z/2\Z)$, this is simply the curve $X=X_{G(m_2)}$. From \cite{zywina2} we know that $X_{G(m_2)}$ is a genus $0$ curve with parameter $t$, such that the map to the $j$-line is as follows.\footnote{Our group $G(m_2)$ is referred to as $G_9$ in \cite{zywina2}.}
\begin{equation}
j(t)=t^3(t^2+5t+40).\label{Eq:XG9}
\end{equation}

The next step is to define a universal family of elliptic curves $E_t$ over $K(t)$, and then find two cubic polynomials over $K(t)$ that generically generate the $S_3$ subextensions of $K(E_t[2])$ and $K(E_t[5])$, respectively. 
A convenient family can be found by substituting $j(t)$ into the following universal family over the $j$-line.
\begin{equation}\label{Eq:EllipticCurveOverJ}
y^2=x^3+\tfrac{1}{4}x^2-\tfrac{36}{j-1728}x-\tfrac{1}{j-1728}
\end{equation}
After a linear change of variables over $\Q$, we arrive at the family $E_t$ given by $y^2=x^3+B(t)x+C(t)$, where $B(t)$ and $C(t)$ are as follows.
\begin{equation}\label{Eq:EllipticCurveLevel10}
B(t)=-3(t - 3) t(t^2 + 5t + 40)\qquad C(t)=2 (t - 3)^2 (t^2 + 4t + 24)(t^2 + 5t + 40)
\end{equation}
The cubic polynomial that generically generates the $S_3$ subextension of $K(E_t[2])$, i.e., the full $2$-torsion field of $E_t$, is simply the Weierstrass polynomial.
In the next lemma, we determine a cubic polynomial over $\Q(t)$ that generically generates the $S_3$ subextension of $K(E_t[5])$.
Thus we are in position to apply Lemma \ref{Lemma:DiagonalCubic} to determine first a singular equation for the genus $0$ entanglement curve, $\cX$, and then a parameter over $\Q$.

\begin{lemma}\label{Lemma:S3Extension-Exceptional5}
The $S_3$ subextension of $K(E_t[5])/K$ is (generically) generated by the roots of the cubic polynomial, $x^3+E(t)x+F(t)$, where
\begin{align*}
E(t) &= -3\left(t^2 + 5t + 40\right)\\
F(t) &= -2\left(t + \tfrac{5}{2}\right)\left(t^2 + 5t + 40\right).
\end{align*}
\end{lemma}
\begin{proof}
Recall that $\ker\psi_2\subseteq \mc{N}_s(5)\subseteq G(m_2)$, with indices of $2$ and $3$, respectively. 
Therefore, the $S_3$ subextension of $K(E_t[5])$ that is determined by $\ker\psi_2$ must be generated (generically) by the natural extension from $X_{G(m_2)}$ up to $X_s^+(5)=X_{\mc{N}_s(5)}$.
More precisely and in the language of Section \ref{Section:GeneralYoga}, the function field of the modular curve, $Y_{S_3,2}$, in this case, is just the normal closure of the function field of $X_s^+(5)$ in the function field of $X(5)$, once $X_s^+(5)$ is viewed as a degree $3$ extension of $X_{G(m_2)}$.

So, essentially, we just need to find an explicit equation for the natural projection from $X_s^+(5)$ to $X_{G(m_2)}$.
One way to do this is to think of the desired extension as an irreducible component of $X_s^+(5)\times_{X(1)}X_{G(m_2)}$ that lies over $X_{G(m_2)}$ with degree $3$ by the canonical map (projection onto the second factor). The $j$-map for the genus $0$ curve, $X_s^+(5)$, is also given in \cite{zywina2} and copied below for convenience.
\begin{equation}
j(s)=\frac{(s+5)^3\left(s^2-5\right)^3\left(s^2+5s+10\right)^3}{\left(s^2+5s+5\right)^5}\label{Eq:XNs5}
\end{equation}
Setting $j(s)=j(t)$ to compute the fiber product, we then factor to find two irreducible components, which lie over $X_{G(m_2)}$ with degrees $3$ and $12$. 
The former is given by
$$s^3 + (-t + 5)s^2 + (-5t - 5)s - 5t - 25=0.$$
The substitution, $x=3s-t+5$, yields the polynomial that is given in the statement of the lemma.
\end{proof}

\begin{Theorem}\label{Thm:Entanglement-Exceptional5}
Let $\cX$ be the modular curve of level $10$ whose $K$-rational points correspond generically to elliptic curves $E/K$ satisfying:\\ 
\\
\indent (1) $\Gal(K(E[2])/K)\cong\GL_2(\Z/2\Z)$\\
\indent (2) $\Gal(K(E[5])/K)\cong \XG$ (from above)\\
\indent (3) $\Gal(K(E[5])\cap K(E[2])/K)\cong S_3$.\\
\\
Then $\cX$ is a genus $0$ curve with a parameter $u$ (over $\Q$) that lies over $X_{G(m_2)}$ via the map given below.
\begin{equation*}\label{Eq:Entanglement-Exceptional5-MapDown}
t=\frac{3u^6 + 12u^5 + 80u^4 + 50u^3 - 20u^2 - 8u + 8}{(u - 1)^2\left(u^2 + 3u + 1\right)^2}
\end{equation*}
\end{Theorem}

\begin{proof}
When we begin to apply the construction of Lemma \ref{Lemma:DiagonalCubic} to the cubic polynomial from the previous lemma and the Weierstrass polynomial of $E_t$, we find that 
$$\delta^2=2^8\cdot 3^{12}\cdot 5(t - 3)^3\left(t^2 + 5t + 40\right)^4.$$
For simplicity, we make the substitution, $\delta=2^4\cdot 3^6(t-3)\left(t^2+5t+40\right)^2y$. Then $y$ is a parameter on the genus $0$ modular curve (lying over $X_{G(m_2)}$) whose $K$-rational points correspond generically to elliptic curves $E/K$ for which  $E[2]$ and $E[5]$ have the desired quadratic entanglement. The map from this curve to $X_{G(m_2)}$ is given by $y^2=5(t-3)$.

Continuing on with Lemma \ref{Lemma:DiagonalCubic}, we then compute the coefficients of the cubic equation, $x^3+Gx^2+Hx+I=0$, over $\Q(t,\delta)=\Q(y)$, which describes the full $S_3$ entanglement modular curve. After making the simplifying substitution, $x=3\cdot 5^{-3}y\left(y^2-5y+40\right)x_0$, we arrive at the equation, $x_0^3+H_0x_0+I_0=0$, where
\begin{align*}
H_0(y)&=-3\left(y^2 + 15\right)\left(y^2 + 5y + 40\right)^2\\
I_0(y)&=\left(y^2 + 5y + 40\right)^2\left(2y^5 + 10y^4 + 125y^3 + 225y^2 + 1125y - 3375\right).
\end{align*}
It is easy to check that the equations given below define a map from $\mbp^1$ (with parameter $u$) to this singular curve.
\begin{align*}
y&=\frac{-5\left(4u^2 + 2u - 1\right)}{(u - 1)\left(u^2 + 3u + 1\right)}\\
x_0&=\frac{25\left(2u^2 + u + 2\right)^2\left(3u^5 + 10u^4 + 25u^3 + 10u^2 + 2\right)}{(u - 1)^3\left(u^2 + 3u + 1\right)^3}
\end{align*}
The map must be a birational isomorphism, as $y$ defines a degree $3$ function on both curves. Composing with $t=\tfrac{1}{5}y^2+3$ yields the formula for $t$ in terms of $u$ that is given in the statement of the theorem.
\end{proof}

\begin{Example}
If we substitute $u=0$ into Theorem \ref{Thm:Entanglement-Exceptional5}, we arrive at $j=73728$ and the following elliptic curve.
$$E: y^2=x^3 - 120x + 500$$
Let $p(x)$ be the $5$-torsion polynomial of $E$, which has degree $12$. Then $p(x)$ is irreducible over $\Q$, and its splitting field is a degree $48$ extension. Adjoining the corresponding $y$ coordinate for any particular root of $p(x)$ generates a further quadratic extension. Since $\Gal(\Q(E[5])/\Q)\subseteq \XG$ (up to conjugation), and the order of $\XG$ is $96$, this confirms that $\Gal(\Q(E[5])/\Q)\cong \XG$. However, over $\Q(\alpha)$ for any root $\alpha$ of the Weierstrass polynomial, $p(x)$ factors into the product of a degree $4$ polynomial and a degree $8$ polynomial.  Hence, we must have $\Q(\alpha)\subseteq \Q(E[5])$. But $\Q(E[2])$ is just the Galois closure of $\Q(\alpha)$. Therefore, since the intersection of Galois extensions must be Galois, it follows that $\Q(E[2])\subseteq\Q(E[5])$, i.e, $\Gal(\Q(E[2])\cap \Q(E[5])/\Q)\cong S_3$.
\end{Example}
\subsection{Level $15$}\label{Section:Level15}

Let $G(m_1)=\GL_2(\Z/3\Z)$, and let $G(m_2)=\XG\subseteq\GL_2(\Z/5\Z)$ be the same subgroup as in Section \ref{Section:Level10}.
Each $G(m_i)$ surjects via some $\psi_i$ onto $S_3$, so that $G:=\pi_{\GL_2}^{-1}\left(G(m_1)\times_{\psi}G(m_2)\right)$ is conjugate to the group $G_{15}$ in Theorem \ref{mainthm}. In this section we derive an explicit equation for the modular curve $\cX:=X_G$ that corresponds to this scenario.
All of the essential information on the $5$-side carries over directly from the previous section.
In particular, there is a genus $0$ modular curve, $X_{G(m_2)}$, whose $K$-rational points correspond generically to elliptic curves $E/K$ for which $\Gal(K(E[5])/K)\cong G(m_2)$. 
The map from $X_{G(m_2)}$ (with parameter $t$) to the $j$-line is given in \eqref{Eq:XG9}, and we have a universal family of elliptic curves $E_t$ over $K(t)$, which is described by \eqref{Eq:EllipticCurveLevel10}.
In the language of Section \ref{Section:GeneralYoga}, we may once again view $X=X_{G(m_2)}$ as the full product curve, since $G(m_1)=\GL_2(\Z/3\Z)$.

The cubic polynomial over $K(t)$ that generically generates the $S_3$ subextension of $K(E_t[5])$ was derived in Lemma \ref{Lemma:S3Extension-Exceptional5}. 
On the other hand, we have the classical result that for elliptic curves $E/K$ with $\Gal(K(E[3])/K)\cong \GL_2(\Z/3\Z)$, the $S_3$ subextension of $K(E[3])$ is the splitting field of $x^3-j$. Therefore, in order to find an explicit model for the entanglement modular curve in this case, we work over $K(t)$ and apply Lemma \ref{Lemma:DiagonalCubic}, using the cubic polynomial from Lemma \ref{Lemma:S3Extension-Exceptional5} and the cubic polynomial $x^3-j(t)$ (where $j(t)$ is as given in \eqref{Eq:XG9}).
Note that we already know, a priori, when the {\em quadratic} subfields of $K(E_t[3])$ and $K(E_t[5])$ coincide.
The two quadratic subfields are $K(\sqrt{-3})$ and $K(\sqrt{5})$, respectively.
Hence, they will coincide if and only if $\sqrt{-15}\in K$.

\begin{Theorem}\label{Thm:Entanglement-Level15}
Let $\cX$ be the modular curve of level $15$ whose $K$-rational points correspond generically to elliptic curves $E/K$ satisfying:\\ 
\\
\indent (1) $\Gal(K(E[3])/K)\cong\GL_2(\Z/3\Z)$\\
\indent (2) $\Gal(K(E[5])/K)\cong \XG$ (from above)\\
\indent (3) $\Gal(K(E[5])\cap K(E[3])/K)\cong S_3$.\\
\\
Then $\cX$ is a genus $0$ curve with a parameter $u$ over $K(\sqrt{-15})$ such that the map to $X_{G(m_2)}$ is as given below.
\begin{equation*}
t=u^3-\tfrac{5-3\sqrt{-15}}{2}
\end{equation*}
\end{Theorem}

\begin{proof}
We begin by computing the discriminants of the two cubic polynomials over $K(t)$.
$$\Delta_1=-3^3\cdot t^6(t^2 + 5t + 40)^2\qquad \Delta_2=3^6\cdot 5\cdot(t^2 + 5t + 40)^2$$
Then the first step in applying Lemma \ref{Lemma:DiagonalCubic} is to adjoin $\delta$ to $\Q(t)$, where
$$\delta^2=\Delta_1\Delta_2=-3^9\cdot 5\cdot t^6(t^2 + 5t + 40)^4.$$
This clearly implies, as was noted above, that $\Q(t,\delta)=\Q(t,\sqrt{-15})$, and so we may continue by taking $\delta=3^4\sqrt{-15}\cdot t^3(t^2 + 5t + 40)^2$.
Applying Lemma \ref{Lemma:DiagonalCubic}, we arrive at the model,
$$x^3-3^3\cdot t^3\left(t +\tfrac{5-3\sqrt{-15}}{2}\right)^2 \left(t + \tfrac{5+3\sqrt{-15}}{2}\right)^3=0.$$
The model given in the statement of the theorem can be obtained by letting
$$x=3\cdot t\left(t +\tfrac{5-3\sqrt{-15}}{2}\right)\left(t + \tfrac{5+3\sqrt{-15}}{2}\right)u^{-1}.$$
\end{proof}

\subsection{Level $18$}

Let $G(m_1)=\GL_2(\Z/2\Z)$ and let $G(m_2)$ be the full pre-image of the Borel group $\tB(3)$ under the canonical projection from $\GL_2(\Z/9\Z)$ onto $\GL_2(\Z/3\Z)$. Then $G(m_2)$ has an index $6$ normal subgroup $N(m_2)$ consisting of all upper triangular invertible matrices for which the diagonal entries are congruent mod $3$. Moreover, the quotient is isomorphic to $S_3$. Fixing a surjection $\psi_2:G(m_2)\twoheadrightarrow S_3$ with $\ker\psi_2=N(m_2)$, and an isomorphism $\psi_1:G(m_1)\xrightarrow{\sim} S_3$, we arrive at a group $G:=\pi_{\GL_2}^{-1}\left(G(m_1)\times_{\psi} G(m_2)\right)$ which is conjugate to the group $G_{18}$ in the statement of Theorem \ref{mainthm}. In this section we derive an explicit equation for the corresponding genus $0$ modular curve $\cX:=X_G$.

Following the yoga, we want to build $\cX$ as an extension of the full product curve $X:=X_{G(m_1),G(m_2)}$, but this is again just $X_{G(m_2)}$ since $G(m_1)=\GL_2(\Z/2\Z)$. Hence, $X$ is canonically isomorphic to the well-known modular curve $X_0(3)$, for which we may choose the following parameter and map to the $j$-line.
$$t=\left(\frac{\eta_1}{\eta_3}\right)^{12}\qquad j(t)=\frac{(t+27)(t+243)^3}{t^3}$$
Substituting $j(t)$ into \eqref{Eq:EllipticCurveOverJ} as before, and making a linear change of variables, we arrive at the family $E_t$ of elliptic curves over $X$ given by $y^2=x^3+B(t)x+C(t)$, where
$$B(t) = -3(t + 27)(t + 243)\qquad C(t) = 2(t + 27)(t^2 - 486t - 19683).$$

Our next step is to determine the two cubic polynomials with coefficients in $\Q(t)$ that generically generate the corresponding $S_3$ subextensions of $K(E_t[2])$ and $K(E_t[9])$ over $K$, respectively, for a given $K$-rational point of $X$.
The first is simply the Weierstrass polynomial, while the second is addressed in the following lemma.

\begin{lemma}\label{Lemma:S3Extension-Level18}
The $S_3$ subextension of $K(E_t[9])/K$ which is fixed by $N(m_2)$ (as above) is generically generated by the roots of the cubic polynomial, $x^3+E(t)x+F(t)$, where
\begin{align*}
E(t) &= -3 t (t + 27)\\
F(t) &= -t (2t + 27)(t + 27).
\end{align*}
\end{lemma}
\begin{proof}
Note that $N(m_2)$, the kernel of the map from $G(m_2)$ to $S_3$, is contained in the Borel group $\tB(9)$. Hence, the corresponding $S_3$ subextension of $\Gal(K(E_t[9])/K)$ will be (generically) generated by a certain degree $3$ factor in the fiber product of $X_{G(m_2)}$ with $X_0(9)$ over the $j$-line. Identifying $X_{G(m_2)}$ with $X_0(3)$, it is the factor whose points correspond in moduli-theoretic terms with triples $(E,C,D)$, where $D$ is cyclic of order $9$ and $C=3D$. In order to determine this factor explicitly, we need an explicit parameter on the genus $0$ modular curve $X_0(9)$, along with an equation for the map to the $j$-line. One choice of parameter is given by the eta product function $s=\left(\eta_1/\eta_9\right)^3$, for which the map is as follows.
$$j(s)=\frac{(s + 9)^3\left(s^3 + 243s^2 + 2187s + 6561\right)^3}{s^9\left(s^2 + 9s + 27\right)}$$
Factoring $j(s)-j(t)$, we find a unique factor of degree $3$ over $K(t)$.
$$s^3 - ts^2 - 9ts - 27t$$
So the roots of this polynomial in $s$ would indeed (generically) generate the desired $S_3$ subextension of $K(E_t[9])$ over $K$. The substitution, $s=\tfrac{1}{3}(x+t)$, yields the equivalent cubic given in the statement of the lemma.
\end{proof}

\begin{Theorem}\label{Thm:Entanglement-Level18}
Let $\cX$ be the modular curve of level $18$ whose $K$-rational points correspond (generically) to elliptic curves $E/K$ satisfying:\\ 
\\
\indent (1) $\Gal(K(E[2])/K)\cong\GL_2(\Z/2\Z)$\\
\indent (2) $\Gal(K(E[9])/K)\cong G(m_2)$ (the full pre-image of $\tB(3)$ in $\GL_2(\Z/9\Z)$)\\
\indent (3) $\Gal(K(E[9])\cap K(E[2])/K)\cong S_3$ (fixed field of $N(m_2)$ as above)\\
\\
Then $\cX$ is a genus $0$ curve with a parameter $u$ (over $\Q$), such that the map to $X_{G(m_2)}$ is given by
\[
t=-27\left(u^3-1\right)^{-2}.
\]
\end{Theorem}

\begin{proof}
In order to apply Lemma \ref{Lemma:DiagonalCubic}, we first set $\delta^2$ equal to the product of the discriminants of the cubic polynomial in Lemma \ref{Lemma:S3Extension-Level18} and the Weierstrass polynomial of $E_t$. Then $\delta$ generates the quadratic entanglement curve over $X$.
$$\delta^2=- 2^8 \cdot 3^{15}(t + 27)^4 t^5$$
If we set $\delta=2^4\cdot 3^7(t+27)^2t^2y$, this simplifies to $y^2=-3t$, so that $y$ is clearly a parameter (over $\Q$) for the genus $0$ curve.

Now that we have $\delta$, we are able to apply Lemma  \ref{Lemma:DiagonalCubic} to obtain an initial singular equation for $\cX$ of the form, $x^3+H(y)x+I(y)=0$.
After making the substitution, $x=-\frac{1}{3}(y+9)x_0$, we arrive at the equation, $x_0^3+H_0(y)x_0+I_0(y)=0$, where $H_0(y)$ and $I_0(y)$ are as follows.
\begin{align*}
H_0(y)&=-3 y^2 (y - 27) (y + 27) (y - 9)^2\\
I_0(y)&=-y^2 (y - 9)^2 (2y^5 - 18y^4 + 2997y^3 - 32805y^2 - 177147y + 1594323)
\end{align*}
It is easy to check that a birational isomorphism over $\Q$ from $\mbp^1$ to this singular curve is given by the following equations.
$$y=\frac{-9}{u^3-1}\qquad x_0=\frac{729u^2\left(3u^5 - 3u^3 - 4u^2 + 2\right)}{\left(u^3-1\right)^3}$$
Composing with $t=-\tfrac{1}{3}y^2$ results in the formula for the forgetful map from $\cX$ to $X_{G(m_2)}$ that is given in the statement of the theorem.
\end{proof}

\bigskip


\section{An infinite family of $D_6$-entanglements} \label{infiniteD6family}

In this section, we exhibit an infinite family of $D_6$ entanglements, which in particular demonstrates that, for fixed $G_0 \in \mc{G}_{\nonab}(0)$, the set
\begin{equation} \label{setofGL2levels}
\{ m_{\GL_2}(G) : G \in \mc{G}_{\nonab}(0), \; G \doteq_{\SL_2} G_0 \}
\end{equation}
is in general unbounded.  First, let $G_3 \subseteq \GL_2(\hat{\mbz})$ be defined by
\[
G_3 := \left\{ g \in \GL_2(\hat{\mbz}) : \pi_3(g) \in \left\{ \begin{pmatrix} * & * \\ 0 & * \end{pmatrix} \right\} \right\}, 
\]
where, here and in what follows, we are denoting by $\pi_m : \GL_2(\hat{\mbz}) \rightarrow \GL_2(\mbz/m\mbz)$ the canonical projection map.
Next, fix an arbitrary fundamental discriminant $D \in \mbz$ and define
\[
\chi_{D} : \GL_2(\hat{\mbz}) \longrightarrow \{ \pm 1 \}, \quad\quad \chi_{D}(g) := \left( \frac{D}{\det g} \right).
\]
We fix isomorphisms 
\begin{equation} \label{keylevel3D6isomorphisms}
\pi_3(G_3) \simeq S_3 \times \{ \pm 1 \}, \quad \pi_2(G_3) \simeq S_3, 
\end{equation}
and define the fibering maps $\psi_3$ and $\psi_{D}$ by
\begin{equation*} 
\begin{tikzcd}
\psi_3 : G_3 \rar{\pi_3} & \pi_3(G_3) \rar{\simeq} & S_3 \times \{ \pm 1 \} \\
\psi_{D} : G_3 \rar{\pi_2 \times \chi_{D}} & \pi_2(G_3) \times \{ \pm 1 \} \rar{\simeq} & S_3 \times \{ \pm 1 \};
\end{tikzcd}
\end{equation*}
we note that $\psi_{D}$ is surjective, provided $D \neq 1$.  Finally, we define the open subgroup $G_{6,D} \subseteq \GL_2(\hat{\mbz})$ by
\[
G_{6,D} := \left\{ g \in G_3 : \; \psi_3(g) = \psi_{D}(g) \right\}.
\]
It is straightforward to see that, under \eqref{keylevel3D6isomorphisms}, we have
\[
\begin{split}
\psi_3\left( G_3 \cap \SL_2(\hat{\mbz}) \right) &= A_3 \times \{ \pm 1 \}, \\
\psi_{D}\left( G_3 \cap \SL_2(\hat{\mbz}) \right) &= S_3 \times \{ 1 \},
\end{split}
\]
and it follows from this that
\[
G_{6,D} \cap \SL_2(\hat{\mbz}) = \psi_3 \vert_{\SL_2(\hat{\mbz})}^{-1}\left( A_3 \times \{ 1 \} \right) \cap  \psi_{D}\vert_{\SL_2(\hat{\mbz})}^{-1}\left( A_3 \times \{ 1 \} \right).
\]
Thus, the groups $G_{6,D}$ all have $\SL_2$-level $6$.  Since the $\GL_2$-level of $G_{6,D}$ is $\lcm(6, |D|)$, this example demonstrates that the set \eqref{setofGL2levels} is indeed unbounded.  Furthermore, we note that $-I \notin G_{6,D}$, and that the group $\tilde{G}_{6,D}$ has level $6$.  Since this group does not depend on $D$, let us denote it by $\tilde{G}_6$.

Under what conditions do we have $\rho_E(G_K) \, \dot\subseteq\,  G_{6,D}$?  Define the map
\[
\eta : \pi_3(G_3) \longrightarrow \{ \pm 1 \}, \quad\quad \eta\left( \begin{pmatrix} a & b \\ 0 & d \end{pmatrix} \right) = a \in (\mbz/3\mbz)^\times \simeq \{ \pm 1 \}.
\]
Assume for simplicity that 
\begin{equation} \label{convenienthypothesis}
\mu_3 \not\subseteq K, \quad\quad \sqrt{D} \notin K(\mu_3).
\end{equation}  
Then, for an appropriate choice of the isomorphism $\pi_3(G) \simeq S_3 \times \{ \pm 1 \}$ in \eqref{keylevel3D6isomorphisms}, we have that for any  elliptic curve $E$ over $K$, $\rho_E(G_K) \, \dot\subseteq \, G_{6,D}$ if and only if $E$ admits a $K$-rational isogeny of degree $3$ and also satisfies the three conditions
\[
K(\mu_3) \subseteq K(E[2]), \quad\quad K(E[3]) = K(E[2],\sqrt{D}), \quad\quad K(E[3])^{\ker \eta} = K(\sqrt{D}).
\]
In particular, setting
$
m_D := \lcm(2,|D|),
$
we have that, under the hypothesis \eqref{convenienthypothesis},
elliptic curves $E / K$ with $\rho_E(G_K) \, \dot\subseteq \, G_{6,D}$ have the entanglement $K(E[3]) \subseteq K(E[m_D])$.
Furthermore, since generically we have $\gal(K(E[3])/K) \simeq \pi_3(G_3) \simeq S_3 \times \{ \pm 1 \} \simeq D_6$, this is an example of a $D_6$-entanglement.  

For each fundamental discriminant $D$, there is an elliptic curve $\mc{E}_{D}$ over $\mbq(t)$ satisfying $\rho_{\mc{E}}(G_{\mbq(t)}) \doteq G_{6,D}$.  To describe it, we first define
\[
\begin{split}
A(t) &:= -3t^9(t^3-2)(t^3+2)^3(t^3+4), \\
B(t) &:= -2t^{12}(t^3+2)^4(t^4-2t^3+4t-2)(t^8+2t^7+4t^6+8t^5+10t^4+8t^3+16t^2+8t+4),
\end{split}
\]
and then set
\begin{equation} \label{modelforXsubGtildesub6D}
\mc{E}_{D} : \; y^2 = x^3 + D^2A(t) x + D^3B(t).
\end{equation}
The discriminant $\gD_{\mc{E}_D}(t)$ and $j$-invariant $j_{\mc{E}_D}(t)$ are given by
\begin{equation*} 
\begin{split}
\gD_{\mc{E}_D}(t) &= 2^{12}3^3D^6t^{24}(t+1)^6(t^2-t+1)^6(t^3+2)^8, \\
j_{\mc{E}_D}(t) &= \frac{-27t^3(t^3 - 2)^3(t^3 + 2)(t^3 + 4)^3}{(t+1)^6(t^2-t+1)^6}.
\end{split}
\end{equation*}
By \cite[Theorem 1.6]{danielslozanorobledo}, the elliptic curve $\mc{E}_1 / \mbq(t)$ has the property that $\rho_{\mc{E}_1,6}(G_{\mbq(t)})$ belongs to one of the two index two subgroups of the level 6 group $\tilde{G}_{6}$ corresponding to elliptic curves $E$ over $\mbq$ satisfying $\mbq(E[2]) = \mbq(E[3])$; its twist $\mc{E}_{-3}$ by $\mbq(\sqrt{-3})$ has mod $6$ image belonging to the other such index two subgroup.  Given this, it is straightforward to verify (e.g. by explicitly computing a Galois-stable cyclic subgroup $\mc{C} \subseteq \mc{E}_D[3]$) that $\mc{E}_{D}$ admits a $\mbq(t)$-rational isogeny of degree three and that the three conditions
\[
\mbq(t)\left( \mu_3 \right) \subseteq \mbq(t)\left( \mc{E}_{D}[2] \right), \quad\quad \mbq(t)\left(\mc{E}_{D}[3]\right) = \mbq(t) \left( \mc{E}_{D}[2], \sqrt{D} \right), \quad\quad \mbq(t)\left(\mc{E}_{D}[3]\right)^{\ker \eta} = \mbq(t) \left( \sqrt{D} \right)
\]
hold.  Thus $\rho_{\mc{E}_{D}}(G_{\mbq(t)}) \, \dot\subseteq \, G_{6,D}$, and by examining specializations, we may see that in fact $\rho_{\mc{E}_{D}}(G_{\mbq(t)}) \doteq G_{6,D}$.
\begin{remark}
A curious feature of the underlying group $\tilde{G}_6$ in the above example is that, given any elliptic curve $E$ over $\mbq$ for which $\rho_E(G_\mbq) \subseteq \tilde{G}_6$, we have $-I \notin \rho_E(G_\mbq)$, in spite of the fact that $-I \in \tilde{G}_6$. The reason for this is as follows: a computation shows that
\[
-I \notin \left[ \tilde{G}_6(6), \tilde{G}_6(6) \right].
\]
In the language of Section \ref{Section:Applications}, this implies that there are no commutator-thick subgroups of $\tilde{G}_6$ that contain $-I$. In particular, since $\rho_E(G_\mbq)$ is commutator-thick (see \eqref{whyrhosubEofGQiscommutatorthick}), we conclude that $-I \notin \rho_E(G_\mbq)$.  By the same reasoning, the same conclusion holds for the group $G$ appearing in Theorem \ref{applicationthm}.
\end{remark}
\bigskip


\end{document}